\documentclass[aps,pre,10pt,showpacs,floatfix,onecolumn,nofootinbib,a4paper]{revtex4-2}
\usepackage{amssymb, amsmath, amsthm, mathtools}
\usepackage{bm}
\usepackage{mathrsfs}
\usepackage{hyperref,bookmark}
\usepackage{graphicx}
\usepackage{epsfig,color}
\usepackage{mathrsfs}
\usepackage{verbatim}
\usepackage{url}
\usepackage{subcaption}
\usepackage{float}
\usepackage[greek,english]{babel}
\usepackage{dcolumn}
\usepackage{accents}

\usepackage{cases}

\newcommand{\tr}{\operatorname{tr}}
\newcommand{\dd}{\operatorname{d}\!}

\newcommand{\diver}{\operatorname{div}}
\newcommand{\curl}{\operatorname{curl}}

\newcommand{\n}{\bm{n}}
\newcommand{\e}{\bm{e}}

\newcommand{\bend}{\bm{b}}

\newcommand{\vv}{\bm{v}}
\newcommand{\frameca}{(\e_x,\e_y,\e_z)}

\newcommand{\twist}{w}
\newcommand{\sgn}{\operatorname{sgn}}
\newcommand{\arcsinh}{\operatorname{arcsinh}}
\newcommand{\etal}{\eta_{\mathrm{L}}}

\newcommand{\etau}{\eta_{\mathrm{U}}}

\newcommand{\lambdamax}{\lambda_{\mathrm{max}}}
\newcommand{\tc}{t_{\mathrm{c}}}
\newcommand{\tch}{\widehat{t}_{\mathrm{c}}}
\newcommand{\tast}{t_{\lambda}^\ast}
\newcommand{\wu}{\widehat{u}}
\newcommand{\etaLp}{\eta_\mathrm{L}^+}
\newcommand{\etaLm}{\eta_\mathrm{L}^-}
\newcommand{\class}{\mathcal{C}}
\newcommand{\adm}{\mathcal{A}_\lambda(r_0)}
\newcommand{\admh}{\widehat{\mathcal{A}}_\lambda(r_0)}
\newcommand{\arginf}{\operatorname{arg\,inf}}
\newcommand{\argmin}{\operatorname{arg\,min}}

\newtheorem{theorem}{Theorem}
\newtheorem{definition}{Definition}
\newtheorem{lemma}{Lemma}

\newtheorem{proposition}{Proposition}

\theoremstyle{definition}
\newtheorem{remark}{Remark}

\begin{document}
\latintext

\title{Singular Damped Twist Waves in Chromonic Liquid Crystals}
\author{Silvia Paparini}
\email{silvia.paparini@unipd.it}
\affiliation{Department of Mathematics, University of Padua, Via Trieste 63, 35121 Padova, Italy}
\author{Epifanio G. Virga}
\email{eg.virga@unipv.it}
\affiliation{Department of Mathematics, University of Pavia, Via Ferrata 5, 27100 Pavia, Italy}
\begin{abstract}
	\emph{Chromonics} are special classes of nematic liquid crystals, for which a \emph{quartic} elastic theory seems to be more appropriate than the classical quadratic Oseen-Frank theory. The relaxation dynamics of twist director profiles are known to develop a shock wave in finite time in the inviscid case, where dissipation is neglected. This paper studies the dissipative case. We give a sufficient criterion for the formation of shocks in the presence of dissipation and we estimate the critical time at which these singularities develop. Both criterion and estimate depend on the initial director profile. We put our theory to the test on a class of initial \emph{kink} profiles and we show how accurate our estimates are by comparing them to the outcomes of numerical solutions.
\end{abstract}
\date{\today}

\maketitle

\section{Introduction}\label{sec:intro}
\emph{Chromonic} liquid crystals (CLCs) are special \emph{lyotropic} phases that arise in colloidal solutions (mostly aqueous) when the concentration of the solute is sufficiently high or the temperature is sufficiently low.  These materials, which have potential applications in life sciences \cite{shiyanovskii:real-time,mushenheim:dynamic,mushenheim:using,zhou:living}, are constituted by plank-shaped molecules that arrange themselves in stacks when dissolved in water. For sufficiently high concentrations or low temperatures, the constituting stacks give rise to an ordered phase, either \emph{nematic} or \emph{columnar} \cite{lydon:chromonic_1998,lydon:handbook,lydon:chromonic_2010,lydon:chromonic,dierking:novel}. Here, we shall only be concerned with the nematic phase, in which the elongated  microscopic constituents of the material display  a certain degree of \emph{orientational} order described by the nematic director $\n$, while their centres of mass remain disordered. 

The classical quadratic elastic  theory of Oseen~\cite{oseen:theory} and Frank~\cite{frank:theory} was proved to have potentially paradoxical consequences when applied to free-boundary problems for chromonics, such as those involving the equilibrium shape of CLC droplets surrounded by their isotropic phase \cite{paparini:paradoxes}. To remedy this state of affairs, a \emph{quartic} elastic theory was proposed for CLCs in \cite{paparini:elastic}, which alters the Oseen-Frank energy density by the addition of a \emph{single} quartic term in the \emph{twist} measure of nematic distortion.

Preliminary possible experimental confirmations of the validity of this theory are presented in \cite{paparini:spiralling,ciuchi:inversion,paparini:what}. These were all concerned with equilibrium configurations of the nematic director. In \cite{paparini:singular}, we began studying the relaxation dynamics of these phases, under the simplifying assumption that the fluid remains quiescent; we showed that the quartic structure of the elastic energy is responsible for the formation of shock waves in a \emph{twist} director mode. Only the \emph{inviscid} case was considered in \cite{paparini:singular}, with all dissipative effects neglected. In this paper, which is the natural sequel of \cite{paparini:singular}, \emph{rotational} viscosity is reinstated. 

We wonder whether dissipation may systematically grant regularity to the solutions of the dynamical equations, thus avoiding shock wave formation. To tackle this issue, after recalling from \cite{paparini:singular} our mechanical theory for CLCs in Section~\ref{sec:recap}, we develop in Section~\ref{sec:method} a method that extends to dissipative systems the approach purported in \cite{manfrin:note}, which has been a source of inspiration for our previous work. In Section~\ref{sec:pre-breakdown}, we illustrate the properties exhibited by the solutions to the global Cauchy problem studied in this paper as long as they maintain the same regularity possessed at the initial time. In Section~\ref{sec:critical_time_estimates}, which is the heart of the paper, we establish an estimate for the critical time at which at shock emerges out of a regular solution (when it does). This estimate is articulated in a sufficient criterion for the occurrence of a singularity and a selection rule that delimits the tolerated amount of dissipation. The theory is applied in Section~\ref{sec:applications} to a class of initial director profiles exhibiting a \emph{kink}. Our theoretical estimates for the critical time are compared with the outcomes of a number of numerical solutions of the dynamical problem. Finally, in Section~\ref{sec:conclusion}, we collect the conclusions of this study and comment briefly on possible ways to extend it. The paper is closed by an Appendix containing auxiliary results needed in the main body, but marginal to its development.

\section{Theoretical Recap}\label{sec:recap}
Here, to make this paper self-contained, we recall from \cite{paparini:singular} the main features of our mechanical theory for CLCs; the reader is referred to \cite{paparini:singular} for all missing details.

Letting $\n$ be the unit vector field representing the local orientation of the molecular aggregates constituting CLCs, we write the elastic free-energy density in the form
\begin{equation}
	\label{eq:quartic_free_energy_density}
	W(\n,\nabla\n):=\frac{1}{2}(K_{11}-K_{24})S^2+\frac{1}{2}(K_{22}-K_{24})T^2+ \frac{1}{2}K_{23}B^{2}+2K_{24}q^2 + \frac14K_{22}a^2T^4,
\end{equation}
where $S:=\diver\n$ is the \emph{splay}, $T:=\n\cdot\curl\n$ is the \emph{twist}, $B^2:=\bend\cdot\bend$ is the square modulus of the \emph{bend} vector $\bend:=\n\times\curl\n$, and, in accordance with a reinterpretation of the basic elastic modes proposed in \cite{selinger:interpretation} (see also \cite{selinger:director}), $q>0$ is the \emph{octupolar splay}  \cite{pedrini:liquid} defined by the following equation,
\begin{equation}
	\label{eq:identity}
	2q^2=\tr(\nabla\n)^2+\frac12T^2-\frac12S^2.
\end{equation}
In \eqref{eq:quartic_free_energy_density}, $a$ is a characteristic \emph{length}, which here shall be treated as a phenomenological parameter, and the $K$'s are elastic moduli that must satisfy the following inequalities for $W$ to be bounded below,
\begin{subequations}\label{eq:new_inequalities}
	\begin{eqnarray}
		&K_{11}\geqq K_{24}\geqq0,\label{eq:new_inequalities_1}\\
		&K_{24}\geqq K_{22}\geqq0, \label{eq:new_inequalities_2}\\
		&K_{33}\geqq0.\label{eq:new_inequalities_3}
	\end{eqnarray}
\end{subequations}

These are the form appropriate to this theory of \emph{Ericksen's inequalities}, which were formulated for the classical Oseen-Frank theory \cite{ericksen:liquid}. If they hold, as assumed here, then $W$ attains its minimum for
\begin{equation}
	\label{eq:double_twist}
	S = 0, \quad T = \pm T_0, \quad B = 0, \quad q = 0\quad\text{with}\quad T_0:=\frac{1}{a}\sqrt{\frac{K_{24}-K_{22}}{K_{22}}},
\end{equation}
a distortion state known as \emph{double twist} \cite{selinger:director}.

\emph{Twist} waves in nematic liquid crystals were first studied by Ericksen in \cite{ericksen:twist}. They are special solutions to the hydrodynamic equations under the assumption that  the flow velocity $\vv$ vanishes: this implies that the motion of the director induces no \emph{backflow}. The governing  one-dimensional wave equation derived  in \cite{ericksen:twist} presumes that no extrinsic body forces or couples act on the system, and that the material occupies the whole space, assumptions that will be retained in this paper.

Under these assumptions, as shown in \cite{paparini:singular}, the Cauchy stress tensor $\mathbf{T}$ takes the following form,
\begin{equation}
	\label{eq:momentum_balance_tw}
	\mathbf{T}=-p\mathbf{I}-\left(\nabla\n\right)^{\mathsf{T}}\frac{\partial W}{\partial\nabla\n}+\mu_2\partial_t\n\otimes\n+\mu_3\n\otimes\partial_t\n,
\end{equation}
where $p$ is the unknown pressure associated with the constraint of incompressibility and
\begin{equation}
	\label{eq:mu_2_3}
	\mu_2:=\frac12(\gamma_2-\gamma_1)\quad\text{and}\quad\mu_3:=\frac12(\gamma_2+\gamma_1)
\end{equation}
are \emph{viscosity} coefficients. In particular, the \emph{rotational} viscosity $\gamma_1>0$ will play here a central role.

In the absence of body forces and couples, the governing balance equations reduce to 
\begin{subequations}
	\label{eq:balance_equations_reduced}
	\begin{gather}
		\diver\mathbf{T}=\bm{0},\label{eq:balance_momentum_reduced}\\
		\sigma\partial_{tt}\n+\gamma_1\partial_t\n+\frac{\partial W}{\partial\n}-\diver\left(\frac{\partial W}{\partial\nabla \n}\right)=\mu\n,
	\end{gather}
\end{subequations}
where $\sigma>0$ is the (density) of molecular moment of inertia and $\mu$ is a Lagrange multiplier associated with the unimodularity constraint, $\n\cdot\n\equiv1$.

By representing $\n$ in a Cartesian frame $\frameca$ as
\begin{equation}
	\label{eq:director_tw}
	\n = \cos \twist\e_y + \sin \twist\e_z, 
\end{equation}
where $\twist = \twist(t, x)$ denotes the \emph{twist} angle for  $(t, x) \in [0, \infty) \times \mathbb{R}$, we easily see that $T$ and $q$ are the only (related) distortion measures that do not vanish, 
\begin{equation}
	\label{eq:distortion_measures}
	S=0,\quad T=-\partial_x\twist,\quad B=0,\quad 2q^2=\frac12T^2.
\end{equation}
Equations \eqref{eq:balance_equations_reduced} are  then equivalent to
\begin{subequations}
	\begin{numcases}{}
	\sigma \partial_{tt}\twist-K_{22}\left(1+3a^2(\partial_x\twist)^2\right)\partial_{xx}\twist=-\gamma_1\partial_t\twist,\label{eq:balance_torques_wave}\\
\mu=-\sigma (\partial_t\twist)^2+K_{22}(\partial_{x}\twist)^2\left(1+a^2(\partial_x\twist)^2\right),\label{eq:lagrange_multiplier}\\
	p=-K_{22}(\partial_{x}\twist)^2(1+a^2(\partial_x\twist)^2)+p_0(t),	\label{eq:pressure}
\end{numcases}
\end{subequations}
where $p_0(t)$ is an arbitrary function of time. While equations \eqref{eq:pressure} and \eqref{eq:lagrange_multiplier} determine the Lagrange multipliers associated with the constraints enforced in the theory, \eqref{eq:balance_torques_wave} is the genuine evolution equation of the system, whose solutions thus provide a complete solution to the governing equations.

The following sections will be devoted to the analysis of equation \eqref{eq:balance_torques_wave}. The molecular inertia $\sigma$ is responsible for its hyperbolic character: equation \eqref{eq:balance_torques_wave} becomes parabolic if $\sigma$ vanishes. 

We found it convenient to rescale lengths to $\sqrt{3}a$ and times to 
\begin{equation}
	\label{eq:time_characteristic}
	\tau:=\sqrt{3}a\sqrt{\frac{\sigma}{K_{22}}}.
\end{equation}
Keeping the original names for the rescaled variables $(t,x)$, we write \eqref{eq:balance_torques_wave} as 
\begin{equation}
	\label{eq:wave_equation_characteristic}
	\partial_{tt}\twist-Q^2(\partial_x\twist)\partial_{xx}\twist=-\lambda \partial_{t}\twist \quad\text{for}\quad (t,x)\in[0,+\infty)\times\mathbb{R},
\end{equation}
where 
\begin{equation}
	\label{eq:kernel}
	Q(\xi):=\sqrt{1+\xi^2}
\end{equation}
is the dimensionless  \emph{wave velocity} and $\lambda$ is a dimensionless \emph{damping} parameter defined as
\begin{equation}
	\label{eq:lambda_def}
	\lambda:=a\gamma_1\sqrt{\frac{3}{\sigma K_{22}}}.
\end{equation}

In our scaling, the molecular inertia $\sigma$ affects both $\lambda$ and $\tau$, making the former larger and the latter smaller when it is decreased, so that the director evolution becomes overdamped and (correspondingly) its hyperbolic character applies to an ever shrinking time scale.

In \cite{paparini:singular}, equation \eqref{eq:balance_torques_wave} was studied in the \emph{inviscid} case, where $\lambda=0$; in the following, it will studied in the general dissipative case, where $\lambda>0$.

\section{Mathematical Methodology}\label{sec:method}
In this section, we study the following global Cauchy problem for the function $w(t,x)$,
\begin{subequations}
\label{eq:wave_system} 
\begin{numcases} {}
\partial_{tt}w-Q^2(\partial_{x}w)\partial_{xx}w=-\lambda \partial_{t}w & for $(t,x)\in[0,+\infty)\times\mathbb{R}$,\label{eq:wave_eq}\\
w(0,x)=w_0(x) & for $x\in\mathbb{R}$,\label{eq:wave_initial_configuration}\\
\partial_{t}w(0,x)=0 & for $x\in\mathbb{R}$,\label{eq:wave_initial_velocity}
\end{numcases}
\end{subequations}
where $\lambda>0$ is defined as in \eqref{eq:lambda_def} and $w_0$ is a function of class $\mathcal{C}^2$ such that its derivative $w_0'$ is bounded, but not constant. The function $Q$ in \eqref{eq:kernel} 
is strictly positive for all $\xi \in \mathbb{R}$, ensuring that the system is \emph{strictly hyperbolic}. However, according to a common definition  (see, e.g., \cite[p.\,15]{majda:compressible}), the system is not \emph{genuinely} nonlinear since $Q'(0)=0$. Additionally,
\begin{equation}
\label{eq:weak_non_linearity}
\xi Q'(\xi)>0 \quad \hbox{ for every } \xi\neq0,
\end{equation}
which rules out the possibility of using the method developed in \cite{maccamy:existence}.
\begin{remark}\label{rmk:aim}
Our aim is to show that there there is a whole range of values of $\lambda>0$ for which a \emph{shock} can still arise from a regular solution of the system \eqref{eq:wave_system} within a finite time $\tast$. The time $\tast$ is expected to depend only  on the initial data and the parameter $\lambda$; the shock should manifest itself as a discontinuity in the derivatives  $\partial_{x}w$ and $\partial_{t}w$ across a singular position $x=x^\ast(t)$ traveling in time, while the second derivatives become infinite.
\end{remark}
In the following, we lay out the method employed in this paper to extend the results arrived at in \cite{paparini:singular}. As will be clearer below, the present method is \emph{not} just an adaptation of the one that applied to the inviscid case.

\subsection{Problem Setup}
We start by setting a different symbol for first-order derivatives of the twist angle $\twist$,
\begin{equation}
\label{eq:first_order_unknowns}
u_1(t,x):=\partial_{x}w(t,x), \quad u_2(t,x):=\partial_{t}w(t,x),
\end{equation}
and transform \eqref{eq:wave_system} into the following first-order system, 
\begin{subequations}
\label{eq:first_order_system} 
\begin{numcases} {}
\partial_t u_{1}=\partial_x u_{2},\\
\partial_t u_{2}-Q^2(u_1)\partial_x u_{1}=-\lambda u_2.
\end{numcases}
\end{subequations}
We diagonalize this system by introducing the \emph{Riemann functions} $r$ and $\ell$ defined as
\begin{subequations}
\label{eq:riemann_system} 
\begin{numcases} {}
r(t,x):=u_2(t,x)-L(u_1(t,x)),\\
\ell(t,x):=u_2(t,x)+L(u_1(t,x)),
\end{numcases}
\end{subequations}
where $L$ is a mapping that satisfies
\begin{equation}
\label{eq:L_prime_u1}
L'(u_1)=Q(u_1).
\end{equation}
The system \eqref{eq:first_order_system} can then be rewritten as
\begin{subequations}
\label{eq:riemann_system_Q}
\begin{numcases} {}
\partial_{t}r+Q\partial_{x}r=-\frac{\lambda}{2}(r+\ell) & $(t,x)\in[0,t_\ast)\times\mathbb{R}$, \label{eq:r_syst_Q}\\
\partial_{t}\ell-Q\partial_{x}\ell=-\frac{\lambda}{2}(r+\ell) & $(t,x)\in[0,t_\ast)\times\mathbb{R}$, \label{eq:l_syst_Q}
\end{numcases}
\end{subequations}
where $[0,t_\ast)$ is  the \emph{maximal interval} of classical existence. We need to express $Q$ as a function of $r$ and $\ell$ for \eqref{eq:riemann_system_Q} to acquire the desired form; by setting $L(0)=0$, with no prejudice for the validity of \eqref{eq:L_prime_u1}, the combination of \eqref{eq:riemann_system} and \eqref{eq:kernel} yields the relation
\begin{equation}
\label{eq:L_u_1}
\ell(t,x)-r(t,x)=2L(u_1)=2\int_0^{u_1}Q(\xi)\dd\xi=u_1\sqrt{1+u_1^2}+\hbox{arcsinh}\,u_1.
\end{equation}
By inverting \eqref{eq:L_u_1}, $u_1$ can be expressed in terms of the difference 
\begin{equation}
\label{eq:eta_def}
\eta:=r-\ell
\end{equation}
as
\begin{equation}
\label{eq:u_1_rl}
u_1=\wu_1(\eta):=L^{-1}\left(-\frac{1}{2}\eta\right).
\end{equation}
The graph of $\wu_1(\eta)$ is represented in Fig.~\ref{fig:u_1}.
\begin{figure}[h] 
\begin{subfigure}[c]{0.48\linewidth}
	\centering
	\includegraphics[width=.65\linewidth]{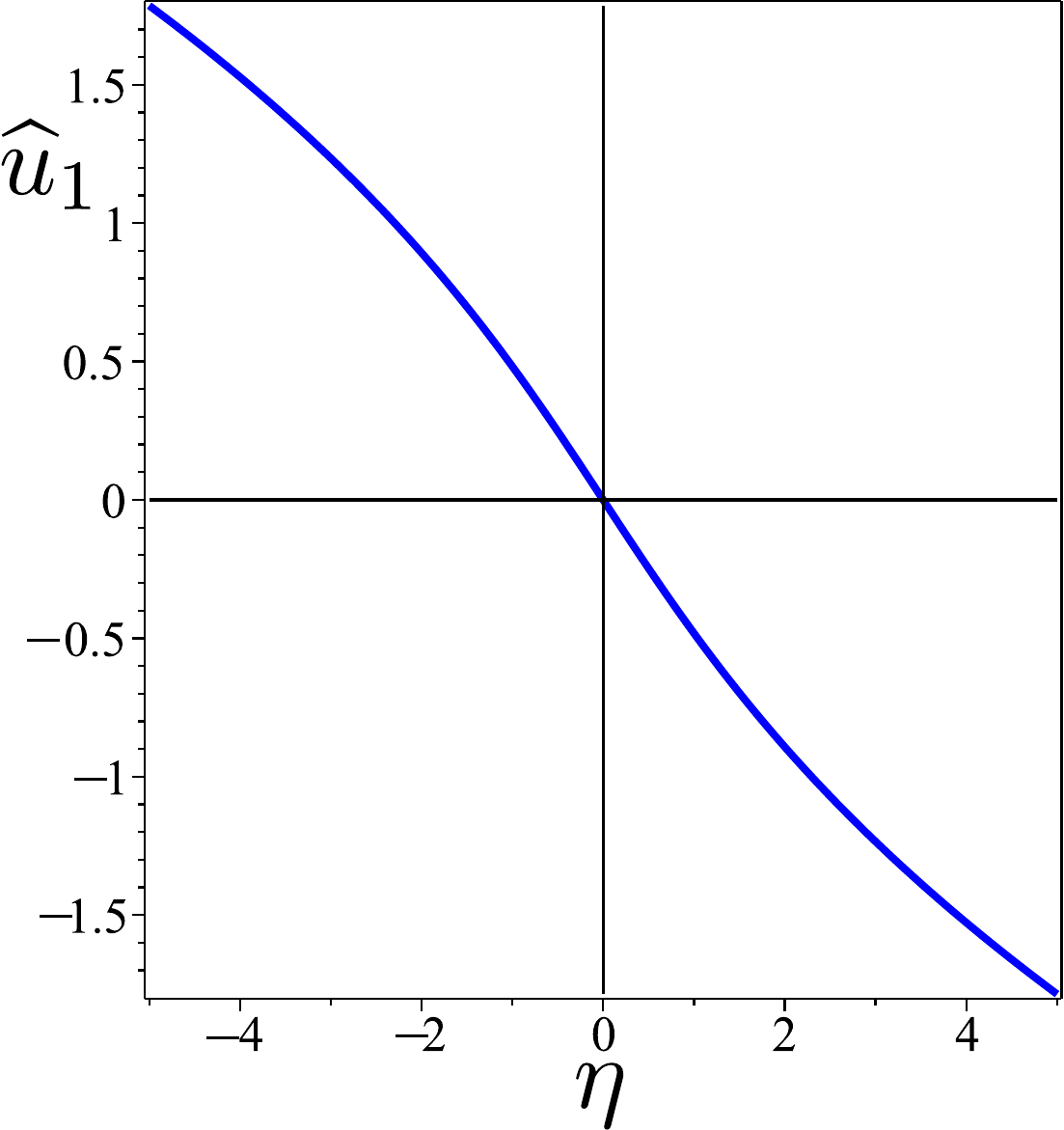}
	\caption{Graphical solution  of \eqref{eq:L_u_1} for $u_1$ depicted as a function of $\eta$.}
	\label{fig:u_1}
	\end{subfigure}
\quad
	\begin{subfigure}[c]{0.45\linewidth}
	\centering
	\includegraphics[width=0.6\linewidth]{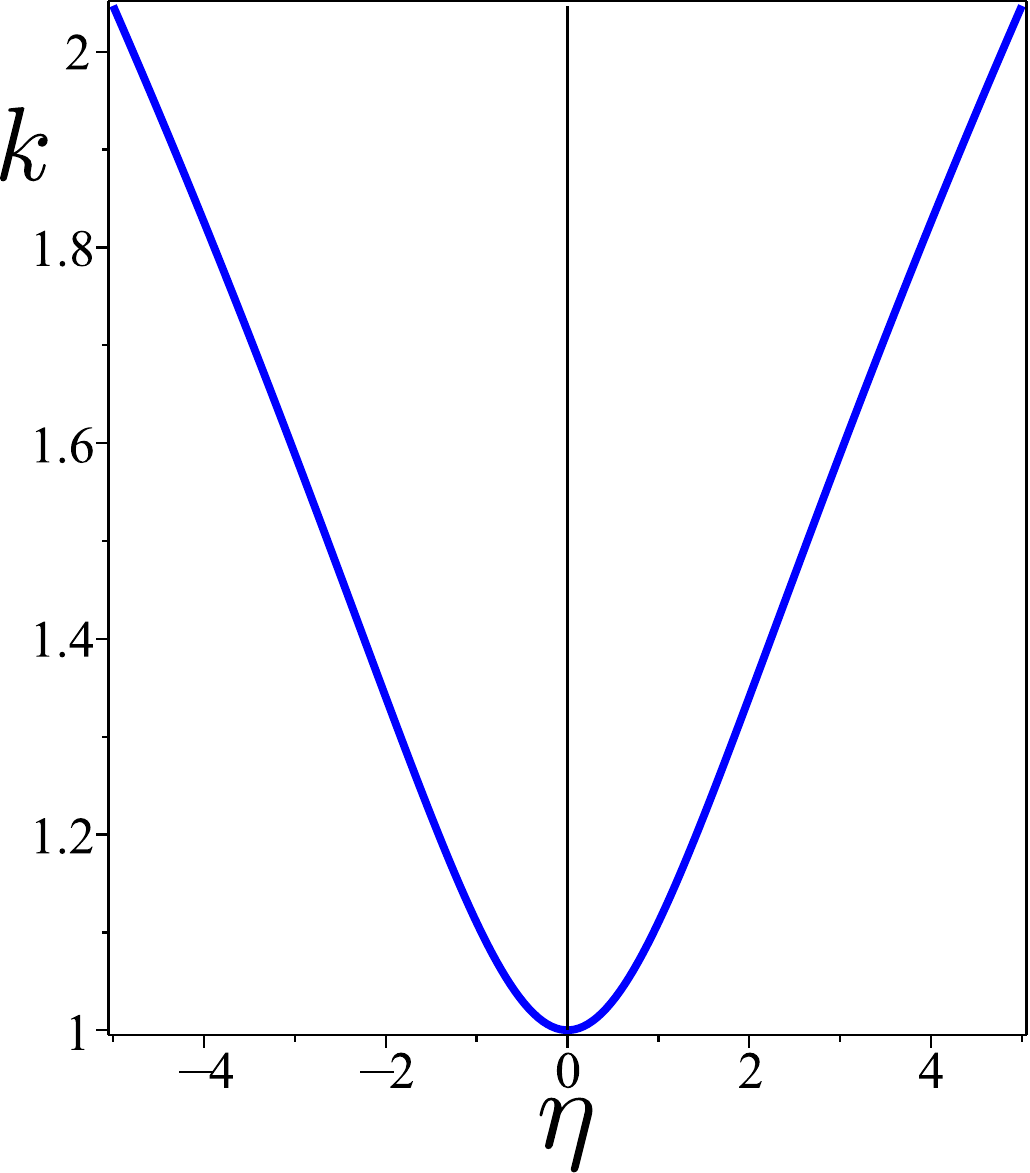}
	\caption{Graph of the function $k$ in \eqref{eq:k_l_r}, illustrating its dependence on $\eta$ through $\widehat{u}_1$ defined in \eqref{eq:u_1_rl}.}
	\label{fig:k_rl}
\end{subfigure}
\caption{Functions $\widehat{u}_1$ and $k$ defined in \eqref{eq:u_1_rl} and \eqref{eq:k_l_r}, respectively, expressed in terms  of $\eta$ as defined by \eqref{eq:eta_def}.}
\label{fig:k_prime_special}
\end{figure}

In \eqref{eq:riemann_system}, we can thus fomally replace the function $Q(u_1)$ with
\begin{equation}
\label{eq:k_l_r}
k(\eta):=Q(\wu_1(\eta)),
\end{equation}
finally arriving at
\begin{subequations}
\label{eq:diagonal_system} 
\begin{numcases} {}
\partial_{t}r+k(r-\ell)\partial_{x}r=-\frac{\lambda}{2}(r+\ell) & $(t,x)\in[0,+\infty)\times\mathbb{R}$, \label{eq:r_syst}\\
\partial_{t}\ell-k(r-\ell)\partial_{x}\ell=-\frac{\lambda}{2}(r+\ell) & $(t,x)\in[0,+\infty)\times\mathbb{R}$, \label{eq:l_syst}
\end{numcases}
\end{subequations}
subject to initial conditions
\begin{equation}
\label{eq:IC_riemann}
r(0,x)=r_0(x)=-L(w_0'(x)), \quad \ell(0,x)=\ell_0(x)=L(w_0'(x))=-r_0(x).
\end{equation}
The graph of the function $k(\eta)$ in \eqref{eq:k_l_r} is illustrated in Fig.~\ref{fig:k_rl}.
\begin{remark}
By \eqref{eq:eta_def}, we can rewrite \eqref{eq:L_u_1} as
\begin{equation}
\label{eq:L_u_1_eta}
2L(\wu_1(\eta))=-\eta.
\end{equation}
By differentiating both sides of \eqref{eq:L_u_1_eta} with respect to $\eta$ and making use of \eqref{eq:L_prime_u1} and \eqref{eq:k_l_r}, we obtain that
\begin{equation}
\label{eq:u_1_eta_diff}
\partial_\eta\wu_{1}(\eta)=-\frac{1}{2k(\eta)}.
\end{equation}
\end{remark}
The \emph{characteristics} of \eqref{eq:diagonal_system} are families of curves $x_1(t,\alpha)$ and $x_2(t,\beta)$, parametrized by $\alpha, \, \beta\in\mathbb{R}$, which, in the inviscid case, represent the trajectories along which the corresponding Riemann functions $r$ and $\ell$ remain constant. More precisely, 
the curve $x_1(t,\alpha)$  starting from a point $\alpha\in\mathbb{R}$ at $t=0$ is the solution of the following differential problem
\begin{equation}
\label{eq:x_1}
\begin{cases}
\frac{\dd x_1}{\dd t}(t,\alpha)=k(\eta(t,x_1(t,\alpha))), \\
x_1(0,\alpha)=\alpha.
\end{cases}
\end{equation}
Similarly, the curve $x_2(t,\beta)$ starting from a point $\beta\in\mathbb{R}$ at $t=0$ solves
\begin{equation}
\label{eq:x_2}
\begin{cases}
\frac{\dd x_2}{\dd t}(t,\beta)=-k(\eta(t,x_2(t,\beta))),\\
x_2(0,\beta)=\beta.
\end{cases}
\end{equation}
\begin{definition}
	\label{def:forward_backward}
	For brevity, we often say that $x_1$ and $x_2$ belong to the \emph{first} and \emph{second} families of characteristics, respectively. 
	Since $k\ge0$, we shall also say that $x_1$ is the \emph{forward} characteristic, whereas $x_2$ is the \emph{backward} characteristic.
\end{definition}
\begin{remark}
	\label{rmk:forward_backward}
The reader should be advised though that this definition is not universally accepted: in equations (3.52) of \cite{majda:compressible}, for example, the role of the two characteristics is interchanged.
\end{remark}

Based on the definition of characteristics $x_1$ and $x_2$, we introduce a concise notation for the derivatives along these curves of any smooth function $f=f(t,x)$. 
\begin{definition}
\label{def:partial_plus_minus}
For any smooth function $f=f(t,x)$, we set
\begin{equation}
\label{eq:partial_plus}
\partial^+f:=\frac{\dd}{\dd t}f(t,x_1(t,\alpha))=\partial_{t}f(t,x_1(t,\alpha))+k(\eta((t,x_1(t,\alpha))))\partial_{x}f(t,x_1(t,\alpha))
\end{equation}
and call it the \emph{forward} derivative of $f$. Similarly, we set 
\begin{equation}
\label{eq:partial_minus}
\partial^-f:=\frac{\dd}{\dd t}f(t,x_2(t,\beta))=\partial_{t}f(t,x_2(t,\beta))-k(\eta(t,x_2(t,\beta)))\partial_{x}f(t,x_2(t,\beta))
\end{equation}
and call it the \emph{backward} derivative of $f$.
\end{definition}
As is apparent from \eqref{eq:diagonal_system}, in the dissipative case the Riemann functions $r$ and $\ell$ in \eqref{eq:riemann_system} fail to be constant along characteristics. Instead, they obey the following laws,
\begin{equation}
\label{eq:r_l_constant}
\begin{cases}
\partial^+r=-\frac{\lambda}{2}\left[r(t,x_1(t,\alpha))+\ell(t,x_1(t,\alpha))\right], \\
\partial^-\ell=-\frac{\lambda}{2}\left[r(t,x_2(t,\beta))+\ell(t,x_2(t,\beta))\right].
\end{cases}
\end{equation}

\begin{remark}\label{rmk:Riemann_invariants_lambda_0}
Classical methods for the inviscid case (see, for example, \cite{chang:existence,maccamy:existence,klainerman:formation}) rely on the invariance of $r$ and $\ell$ along the appropriate characteristic. In the dissipative case, these methods are no longer applicable.
\end{remark}
\begin{remark}\label{rmk:5}
The system \eqref{eq:r_l_constant} governing the evolution of the Riemann functions along characteristics remains locally well-posed under the given initial conditions and assumptions. Indeed, since $w_0'$ is bounded by assumption, the initial values of the Riemann functions, $r_0$ and $\ell_0$, expressed in \eqref{eq:IC_riemann} in terms of $w'_0$, are bounded continuous functions with bounded continuous derivatives. Under these hypotheses, the general theory presented in \cite{douglis:existence} ensures that the strictly hyperbolic system \eqref{eq:diagonal_system} has a unique $\mathcal{C}^1$ solution $(r, \ell)$ locally in time.
\end{remark}

\subsection{Boundedness of Riemann Functions}
In the following, we assume that, for a given value of $\lambda>0$, the pair $(r,\ell)$ is the unique $\mathcal{C}^1$ solution of the system \eqref{eq:diagonal_system} in $[0,t_\ast)\times\mathbb{R}$, subject to the initial conditions in \eqref{eq:IC_riemann} with $r_0\in\class^1$.
Building on the framework recently proposed  in \cite{sugiyama:singularity,sui:vacuum} for dissipative problems, we define
\begin{equation}
\label{eq:A_def}
A(t):=e^{\frac{\lambda}{2}t}
\end{equation}
and rephrase the evolution laws in \eqref{eq:r_l_constant} as
\begin{equation}
\label{eq:A_r_l}
\begin{cases}
\partial^+\left(A(t)r\right)=\frac{\dd}{\dd t}\left(A(t)r(t,x_1(t,\alpha))\right)=-\frac{\lambda}{2}A(t)\ell(t,x_1(t,\alpha)),\\
\partial^-\left(A(t)\ell\right)=\frac{\dd}{\dd t}\left(A(t)\ell(t,x_1(t,\alpha))\right)=-\frac{\lambda}{2}A(t)r(t,x_2(t,\beta)).
\end{cases}
\end{equation}
Integrating both sides of \eqref{eq:A_r_l} and taking into account the initial conditions  in \eqref{eq:IC_riemann}, we obtain that
\begin{equation}
\label{eq:A_r_l_int}
\begin{cases}
A(t)r(t,x_1(t,\alpha))=r_0(\alpha)-\frac{\lambda}{2}\int_0^{t}A(\tau)\ell(\tau,x_1(\tau,\alpha))\dd\tau,\\
A(t)\ell(t,x_2(t,\beta))=\ell_0(\beta)-\frac{\lambda}{2}\int_0^{t}A(\tau)r(\tau,x_2(\tau,\beta))\dd\tau.
\end{cases}
\end{equation}

In the dissipative case, $r(t,x)$ and $\ell(t,x)$ are no longer constant along the characteristics of \eqref{eq:diagonal_system}, but they remain uniformly bounded over time, as we now proceed to show.
\begin{lemma}
\label{lemma:r_l_bounded}
Let $\varphi(t,x,y):=r(t,x)+\ell(t,y)$ and $\eta(t,x,y):=r(t,x)-\ell(t,y)$. Then, for a  solution of the system \eqref{eq:diagonal_system} of class $\mathcal{C}^1$ over $[0,t_\ast)\times\mathbb{R}$, 
\begin{subequations}
\label{eq:bounded_r_l_plus_minus}
\begin{align}
||\varphi(t,\cdot,\cdot)||_{L^{\infty}(\mathbb{R}^2)}&\leq||\varphi(0,\cdot,\cdot)||_{L^{\infty}(\mathbb{R}^2)} \label{eq:bounded_r_l_plus},\\
||\eta(t,\cdot,\cdot)||_{L^{\infty}(\mathbb{R}^2)}&\leq||\eta(0,\cdot,\cdot)||_{L^{\infty}(\mathbb{R}^2)}. \label{eq:bounded_r_l_minus}
\end{align}
\end{subequations}
\end{lemma}
The proof of Lemma~\ref{lemma:r_l_bounded} employs the strategy used in the proof of Lemma~7 in \cite{sugiyama:singularity}; its adaptation to our context is deferred to Appendix~\ref{sec:aux_res}.
\begin{remark}\label{rmk:6}
As a consequence of Lemma~\ref{lemma:r_l_bounded}, for initial profiles $r_0(x)=-\ell_0(x)\in \mathcal{C}^1$ as given in \eqref{eq:IC_riemann}, the following bounds hold on $[0,t_\ast)\times\mathbb{R}$,
\begin{subequations}
\label{eq:bounded_r_l_plus_minus_abs}
\begin{align}
|r(t,x)+\ell(t,y)|\leq||\varphi(t,\cdot,\cdot)||_{L^{\infty}(\mathbb{R}^2)}&\leq||\varphi(0,\cdot,\cdot)||_{L^{\infty}(\mathbb{R}^2)}\leq 2||r_0||_\infty, \label{eq:bounded_r_l_plus_abs}\\
|r(t,x)-\ell(t,y)|\leq||\eta(t,\cdot,\cdot)||_{L^{\infty}(\mathbb{R}^2)}&\leq||\eta(0,\cdot,\cdot)||_{L^{\infty}(\mathbb{R}^2)}\leq 2||r_0||_\infty.\label{eq:bounded_r_l_minus_abs}
\end{align}
\end{subequations}
Here, $||\cdot||_{\infty}$ denotes the $L^\infty$ norm over $\mathbb{R}$. This means that both $r(t,x)$ and  $\ell(t,y)$ are uniformly bounded in $[0,t_\ast)\times\mathbb{R}$.
\end{remark}
\begin{definition}
\label{rmk:t_critical}
	By the \emph{continuation principle} (see, for example, \cite[p.\,100]{majda:compressible} for this particular incarnation), the estimates \eqref{eq:bounded_r_l_plus_minus_abs} imply the following \emph{dichotomy}: for a given value of $\lambda>0$, either there is a \emph{critical time} $\tast$ such that
	\begin{equation}
		\label{eq:dichotomy_estimate}
		||\partial_{x}r||_\infty+||\partial_{x}\ell||_\infty\to+\infty\quad\text{for}\quad t\to \tast,
	\end{equation}
or there is a \emph{global} smooth solution of \eqref{eq:riemann_system} for all $0\leq t<+\infty$. In the former case, which is the one we are interested in, a \emph{shock} wave is formed in a finite time. In the latter case, we conventionally set $\tast=+\infty$.
\end{definition}
\begin{remark}
It follows from the definition of $k$ in \eqref{eq:k_l_r} and the monotonicity of $L$ as expressed by \eqref{eq:L_u_1} that, under the assumptions of Lemma~\ref{lemma:r_l_bounded}, $k$ is subject to the following lower and upper bounds,
\begin{equation}
\label{eq:estimates}
1=k(0)\leq k(\eta)\leq\delta \quad\text{with}\quad \delta:=k\left(||r_0||_{\infty}+||\ell_0||_{\infty}\right)=k\left(2||r_0||_{\infty}\right)>1.
\end{equation}
\end{remark}

\subsection{Energy Dissipation and \emph{Mass} Conservation}
Here, we show that regular solutions of the system \eqref{eq:wave_system} in $[0,\tast)\times\mathbb{R}$ dissipate energy while preserving \emph{mass}.\footnote{The latter is  meant as a synonym for the total twist.} We shall employ the following lemma, which asserts that the limits of the spatial derivative of the twist angle $w(t,x)$, as $x$ approaches $\pm\infty$, coincide with the limits of the derivatives of the initial datum $w_0(x)$.
\begin{lemma}
\label{lemma:w_x_t_infty}
Let $w'_0(\pm\infty):=\lim_{x\to\pm\infty}w_0'(x)$ be finite. Then, for a $\mathcal{C}^1$ solution of the system \eqref{eq:wave_system}, the following limits hold for every $t\in[0,\tast)$,
\begin{equation}
\label{eq:lim_infinity_wx}
\lim_{x\to\pm\infty}\partial_{x}w(t,x)=w_0'(\pm\infty).
\end{equation}  
Similarly, since $\partial_{t}w(0,x)=0$ for all $x\in\mathbb{R}$, then
\begin{equation}
\label{eq:lim_infinity_wt}
\lim_{x\to\pm\infty}\partial_{t}w(t,x)=0,
\end{equation}  
for every $t\in[0,\tast)$.
\end{lemma}
The proof of Lemma~\ref{lemma:w_x_t_infty} can be found in Appendix~\ref{sec:aux_res}; it is an extension to our present context of a similar result proved in Lemma~8 of \cite{sugiyama:singularity}. Building upon Lemma~\ref{lemma:w_x_t_infty}, we derive a dissipation law for the total energy associated with  the system \eqref{eq:wave_system}, which comprises a quadratic kinetic energy and a quartic potential energy.
\begin{proposition}
\label{prop:dissipation_law}
For regular solutions $w(t,x)$ of system \eqref{eq:wave_system}, the following energy dissipation law holds
\begin{equation}
\label{eq:energy_dissipation}
\partial_{t}E(t)=-\lambda\int_{\mathbb{R}}(\partial_{t}w(t,x))^2\dd x,
\end{equation}
where
\begin{equation}
\label{eq:energy_def}
E(t):=\frac{1}{2}\int_{\mathbb{R}}\left[(\partial_{t}w)^2+(\partial_{x}w)^2\left(1+\frac{1}{6}(\partial_{x}w)^2\right)\right]\dd x.
\end{equation}
\end{proposition}
Thus, the energy $E(t)$ of the system decreases over time due to the damping term; the rate of dissipation is proportional to the integral of $(\partial_{t}w(t,x))^2$, which represents the kinetic energy of the system.
\begin{proof}
We multiply both sides of equation \eqref{eq:wave_eq} by $\partial_{t}w(t,x)$ and then integrate by parts with respect to $x$ over $\mathbb{R}$; we thus obtain that
\begin{align}
\label{eq:energy_law0}
\partial_{t}\int_{\mathbb{R}}\frac{1}{2}&\left[(\partial_{t}w(t,x))^2+(\partial_{x}w(t,x))^2\left(1+\frac{1}{6}(\partial_{x}w(t,x))^2\right)\right]\dd x\nonumber\\
&=\left.\left[\partial_{x}w\partial_{t}w\left(1+\frac{1}{3}(\partial_{x}w(t,x))^2\right)\right]\right|_{-\infty}^{+\infty}-\lambda\int_{\mathbb{R}}(\partial_{t}w(t,x))^2\dd x.
\end{align}
By Lemma~\ref{lemma:w_x_t_infty}, the boundary terms in \eqref{eq:energy_law0} vanish for every $t$, which leads us to \eqref{eq:energy_dissipation}.
\end{proof}

\begin{proposition}\label{prop:mass_conservation}
Let a regular solutions $w(t,x)$ of system \eqref{eq:wave_system} be integrable over $\mathbb{R}$ for all $t\in[0,\tast)$, and let $M$ be the function of $t$ defined by
\begin{equation}
\label{eq:mass_def}
M(t):=\int_\mathbb{R}w(t,x)\dd x.
\end{equation}
If the initial datum $w_0(x)$ in \eqref{eq:wave_system} is such that $w_0'(+\infty)=w'_0(-\infty)$, then the following conservation law holds
\begin{equation}
\label{eq:conservation_displacement}
M(t)=M(0) \quad \, \text{for every}\ t\in[0,\tast).
\end{equation}
\end{proposition}
\begin{proof}
By integrating both sides of  equation \eqref{eq:wave_eq} over $\mathbb{R}$, we obtain that
\begin{equation}
\label{eq:displacement_law0}
\partial_{tt}\int_{\mathbb{R}}w(t,x)\dd x=-\lambda\partial_t\int_{\mathbb{R}}w(t,x)\dd x+\left.\left[\partial_{x}w\left(1+\frac{1}{3}(\partial_{x}w)^2\right)\right]\right|_{-\infty}^{+\infty}.
\end{equation}
By Lemma~\ref{lemma:w_x_t_infty}, since  $w'_0(+\infty)=w_0'(-\infty)$, the boundary terms in \eqref{eq:displacement_law0} vanish, and the latter equation reduces to
\begin{equation}
\label{eq:displacement_ode}
\ddot{M}(t)=-\lambda\dot{M}(t)
\end{equation}
for the function $M(t)$ defined in \eqref{eq:mass_def}. By solving \eqref{eq:displacement_ode} with the initial conditions in \eqref{eq:wave_initial_configuration} and \eqref{eq:wave_initial_velocity}, we arrive at \eqref{eq:conservation_displacement}.
\end{proof}

\section{Pre-Breakdown Properties}\label{sec:pre-breakdown}
Here, we prepare the ground for the analysis of the formation of discontinuities in the derivatives of the Riemann functions that will be performed in the following section. Detailed proof of the following preparatory results are deferred to Appendix~\ref{sec:aux_res}.

\subsection{Evolution along characteristics}
We first focus on the behavior of the spatial derivatives of the solutions $r$ and $\ell$ of \eqref{eq:diagonal_system} along the characteristic curves described by \eqref{eq:x_1} and \eqref{eq:x_2}. This is meant to lay the foundation for the analysis of shock formation along characteristics. We are primarily interested in sufficient conditions for the breakdown of solutions.
\begin{definition}
Let the pair $(r,\ell)$ be a solution of class $\mathcal{C}^1$ of \eqref{eq:diagonal_system}. We define the pair of \emph{weighted Riemann derivatives} $(p,q)$ as
\begin{subequations}
\label{eq:y_q_def}
\begin{align}
p(t,x)&:=A(t)\sqrt{k(\eta(t,x))} \, \partial_{x}r(t,x),\label{eq:y_def}\\
q(t,x)&:=A(t)\sqrt{k(\eta(t,x))} \, \partial_{x}\ell(t,x).\label{eq:q_def}
\end{align}
\end{subequations}
\end{definition}
The following Proposition analyzes the behavior of  $(p,q)$ along the characteristic curves $x_1=x_1(t,\alpha)$ and $x_2=x_2(t,\beta)$, respectively.
\begin{proposition}
If the pair $(r,\ell)$ is a solution of class $\mathcal{C}^1$ of \eqref{eq:diagonal_system}, then the pair $(p,q)$  satisfies
\begin{subequations}
\label{eq:y_q_characteristics}
\begin{align}
\partial^+p=&-\frac{\lambda}{2}q(t,x_1(t,\alpha))-A(t)^{-1}f(\eta(t,x_1(t,\alpha)))p(t,x_1(t,\alpha))^2,\label{eq:partial_plus_y}\\
\partial^-q=&-\frac{\lambda}{2}p(t,x_2(t,\beta))-A(t)^{-1}f(\eta(t,x_2(t,\beta)))q(t,x_2(t,\beta))^2,\label{eq:partial_plus_q}
\end{align}
\end{subequations}
where $f$ is the function defined by
\begin{equation}
	\label{eq:f_definition}
	f(\eta):=\frac{k'(\eta)}{\sqrt{k(\eta)}}.
\end{equation}
\end{proposition}
\begin{proof}
Using Definition~\ref{def:partial_plus_minus}, we obtain that
\begin{equation}
\label{eq:partial_x_partial_plus_r}
\partial_{x}\partial^+r=\partial_{tx}r+k(\eta)\partial_{xx}r+k'(\eta)(\partial_{x}r-\partial_{x}\ell)\partial_{x}r
=\partial_{x}\partial^+r+k'(\eta)(\partial_{x}r-\partial_{x}\ell)\partial_{x}r,
\end{equation}
where all quantities are evaluated along the characteristic curve $x_1(t,\alpha)$. By differentiating both sides of the first equation in \eqref{eq:r_l_constant} with respect to $x$ and using \eqref{eq:partial_x_partial_plus_r}, we obtain that
\begin{equation}
\label{eq:partial_plus_partial_x_r}
\partial^+\partial_{x}r=-\frac{\lambda}{2}(\partial_{x}r+\partial_{x}\ell)-k'(\eta)(\partial_{x}r-\partial_{x}\ell)\partial_{x}r.
\end{equation}
Similarly, we have that
\begin{equation}
\label{eq:partial_minus_partial_x_l}
\partial^-\partial_{x}\ell=-\frac{\lambda}{2}(\partial_{x}r+\partial_{x}\ell)-k'(\eta)(\partial_{x}\ell-\partial_{x}r)\partial_{x}\ell,
\end{equation}
where all quantities are evaluated along the characteristic curve $x_2(t,\beta)$.
From \eqref{eq:L_u_1} and \eqref{eq:riemann_system_Q} we can also derive the equations 
\begin{equation}
\label{eq:partial_plus_minus_u1}
\partial^+\eta=-2k(\eta(t,x_1(t,\alpha)))\partial_{x}\ell(t,x_1(t,\alpha)), \quad \partial^-\eta=-2k(\eta(t,x_2(t,\beta)))\partial_{x}r(t,x_2(t,\beta)).
\end{equation}
Then, by multiplying both sides of \eqref{eq:partial_plus_partial_x_r} and \eqref{eq:partial_minus_partial_x_l} by $A(t)\sqrt{k(\eta)}$ evaluated at $(t,x_1(t,\alpha))$ and $(t,x_2(t,\beta))$, respectively, and making use of \eqref{eq:y_q_def} and \eqref{eq:f_definition}, we obtain \eqref{eq:y_q_characteristics}.
\end{proof}
\begin{remark}
	\label{rmk:f_parametric}
	Since $f$ is defined implicitly by \eqref{eq:k_l_r}, it is better represented in parametric form. By differentiating both sides of \eqref{eq:k_l_r}, and using \eqref{eq:kernel}, \eqref{eq:L_prime_u1}, and \eqref{eq:u_1_rl}, we easily arrive at 
	\begin{equation}
		\label{eq:f_parametric}
		\eta=-u\sqrt{1+u^2}-\arcsinh u,\quad f=-\frac{1}{2}\frac{u}{(1+u^2)^{5/4}},
	\end{equation}
the former being simply \eqref{eq:L_u_1} rewritten. It is clear from \eqref{eq:f_parametric} that $f$ is an odd function of $\eta$. It exhibits an isolated minimum at $\eta=\eta_0$ and an isolated maximum at $\eta=-\eta_0$, whose values are $\mp f_0$, respectively, corresponding to the values $u_0=\pm\sqrt{2/3}$ of the parameter $u$ in \eqref{eq:f_parametric}.
\end{remark}
The graph of the function $f(\eta)$ is illustrated in Fig.~\ref{fig:k_prime_rl}. 
\begin{figure}[h] 
	\centering
	\includegraphics[width=0.3\linewidth]{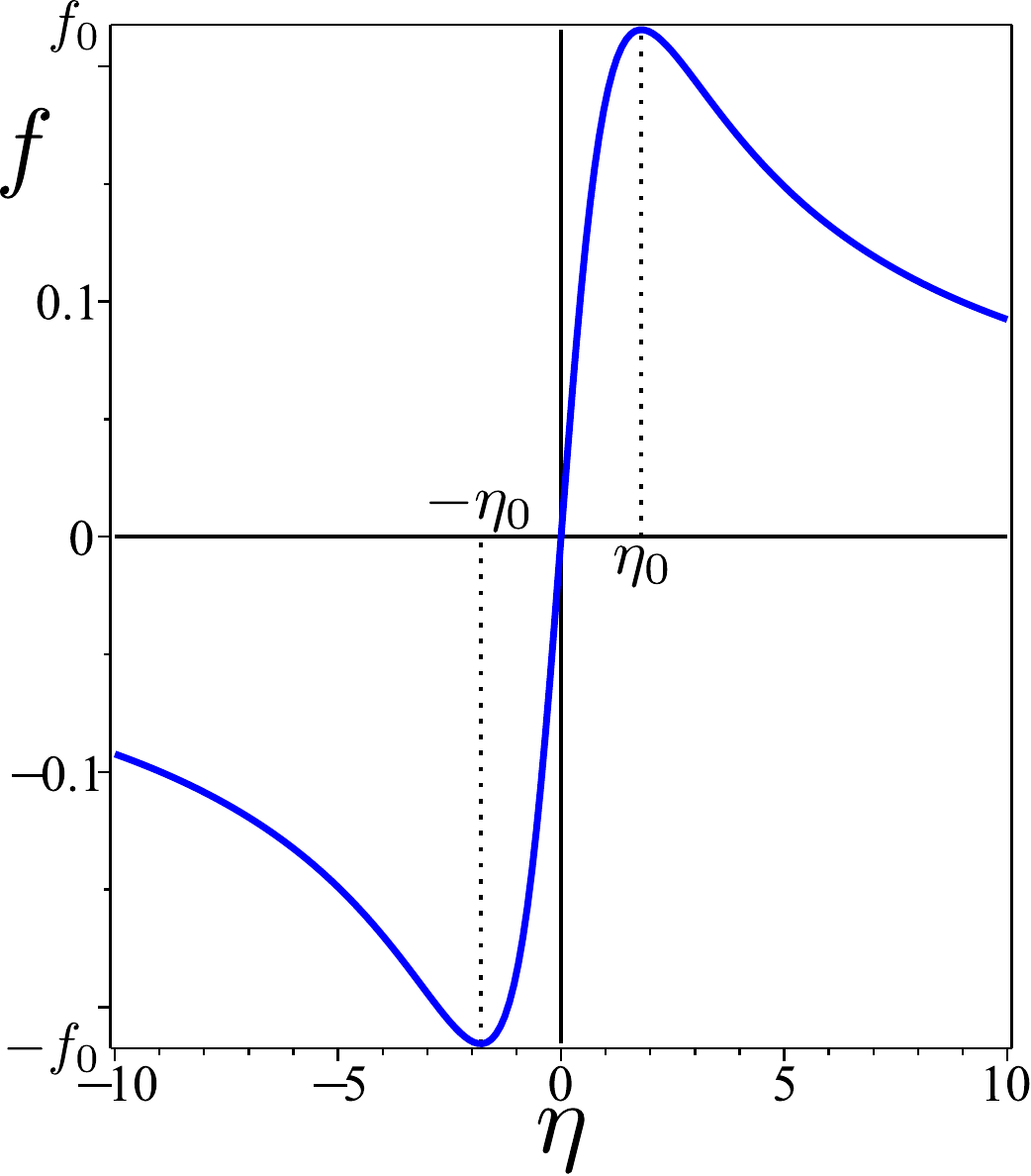}
	\caption{Graph of the odd function $f(\eta)$. It is bounded between $-f_0$ and $+f_0$, which are the values attained by $f$ at $\eta=\mp\eta_0$, respectively. The numerical values are $\eta_0\doteq1.80$ and $f_0\doteq0.22$, which correspond to $u_0=-\sqrt{2/3}$ in \eqref{eq:f_parametric}.}
	\label{fig:k_prime_rl}
\end{figure}

\subsection{Damped Shock Waves}\label{sec:beta}
A critical aspect of our analysis is to identify every point $(t,x)$ reached by a characteristic of one family as the end-point, at time $t$, of a characteristic in the other family. This identification is formalized by the functions $\beta(t,\alpha)$ and $\alpha(t,\beta)$ defined as follows.
\begin{definition}
Every point $(t,x)$ reached by a forward characteristic $x_1(t,\alpha)$ can be seen as the end-point at time $t$ of a backward characteristic originating at $\beta(t,\alpha)$, where $\beta(t,\alpha)$ is formally defined by 
\begin{subequations}\label{eq:alpha_beta_implicit_definitions}
	\begin{equation}
	\label{eq:beta_of_alpha_implicit_definition}
	x_1(t,\alpha)=x_2(t,\beta(t,\alpha)).
\end{equation}
Conversely, by exchanging the roles of forward and backward characteristics, the function $\alpha(t,\beta)$ is defined implicitly by
\begin{equation}
	\label{eq:alpha_of_beta_implicit_definition}
	x_2(t,\beta)=x_1(t,\alpha(t,\beta)).
\end{equation}
\end{subequations}
\end{definition}
Instrumental to the analysis of the mappings $t\mapsto\beta(t,\alpha)$ and $t\mapsto\alpha(t,\beta)$ are the wave \emph{infinitesimal compression ratios}, which can be defined as follows.
\begin{definition}
	\label{def:infinitesimal_compression_ratios}
	For every characteristic curve, $x=x_1(t,\alpha)$ and $x=x_2(t,\beta)$, corresponding to a smooth solution of \eqref{eq:diagonal_system}, the wave \emph{infinitesimal compression ratios}, $c_1$ and $c_2$, are defined by
	\begin{equation}
		\label{eq:infinitesimal_compression_ratios}
		c_1(t,\alpha):=\partial_\alpha x_1\quad\text{and}\quad c_2(t,\beta):=\partial_\beta x_2.
	\end{equation}
\end{definition}
These can be expressed in terms of the pair $(p,q)$ in \eqref{eq:y_q_def} and the function $f$ in \eqref{eq:f_definition}.
\begin{proposition}
\label{th:infinitesimal_compression_ratios}
If the pair $(r,\ell)$ is a local solution of class $\mathcal{C}^1$ of the system \eqref{eq:diagonal_system} in $[0,\tast)\times\mathbb{R}$, then
the wave infinitesimal compression ratios $c_1$ and $c_2$ are given by
\begin{subequations}
\label{eq:infinitesimal_compression_ratios_expr}
\begin{align}
c_1(t,\alpha)=&\sqrt{\frac{k(\eta(t,x_1(t,\alpha)))}{k(2r_0(\alpha))}}\mathrm{e}^{{\Huge \int_0^tA(\tau)^{-1}f(\eta(\tau,x_1(\tau,\alpha)))p(\tau,x_1(\tau,\alpha))\dd\tau}}>0,\label{eq:infinitesimal_compression_ratios1_expr}\\
c_2(t,\beta)=&\sqrt{\frac{k(\eta(t,x_2(t,\beta)))}{k(2r_0(\beta))}}\mathrm{e}^{{\Huge \int_0^tA(\tau)^{-1}f(\eta(\tau,x_2(\tau,\beta)))q(\tau,x_2(\tau,\beta))\dd\tau}}>0.\label{eq:infinitesimal_compression_ratios2_expr}
\end{align}
\end{subequations} 
\end{proposition}
The proof of Proposition~\ref{th:infinitesimal_compression_ratios} is given in Appendix~\ref{sec:aux_res}. 
\begin{remark}
	\label{rkm:no_overlapping}
	As long as both $r$ and $\ell$ are of class $\mathcal{C}^1$, inequalities \eqref{eq:infinitesimal_compression_ratios_expr} remain valid, and \emph{vice versa}, meaning that characteristics in the same family do \emph{not} crash on one another. 
\end{remark}
Since both $c_1$ and $c_2$  are always positive for regular solutions, we can then derive the following properties for the mappings $\beta=\beta(t,\alpha)$ and $\alpha=\alpha(t,\beta)$. This result is proved in Appendix~\ref{sec:aux_res}
\begin{proposition}
\label{prop:beta_t_alpha}
For any fixed $\alpha\in\mathbb{R}$ and $\beta\in\mathbb{R}$, equations \eqref{eq:alpha_beta_implicit_definitions} have a unique solution, $\beta=\beta(t,\alpha)$ and $\alpha=\alpha(t,\beta)$, respectively. The mappings $t\mapsto\beta(t,\alpha)$ and $t\mapsto\alpha(t,\beta)$ are both of class $\mathcal{C}^1$, strictly increasing and strictly decreasing, respectively, and such that
\begin{equation}
\label{eq:d_t_beta}
\partial_{t}\beta(t,\alpha)=\frac{2 k(\eta(t,x_1(t,\alpha)))}{c_2(t,\beta(t,\alpha))}>0 \quad \hbox{ and } \quad \partial_{t}\alpha(t,\beta)=-\frac{2k(\eta(t,x_2(t,\beta)))}{c_1(t,\alpha(t,\beta))}<0
\end{equation}
Moreover, they obey the following bounds
\begin{equation}
\label{eq:beta_confined}
\alpha\leq\beta(t,\alpha)\leq\alpha+2\delta t, \quad \hbox{and} \quad \beta-2\delta t\leq\alpha(t,\beta)\leq\beta,
\end{equation}
where $\delta$ is defined in \eqref{eq:estimates}.
\end{proposition}
\begin{remark}
	\label{rmk:alpha_beta_nesting}
	It readily follows from \eqref{eq:beta_confined} that $\beta(0,\alpha)=\alpha$ and  $\alpha(0,\beta)=\beta$.
\end{remark}
\begin{definition}
	\label{def:degeneracy}
	For a given value of $\lambda>0$, we say that a \emph{shock wave} is formed when either $p$ or $q$ in \eqref{eq:y_q_def} diverges along one family of characteristics. Specifically, this requires either that there is $\alpha\in\mathbb{R}$ and a finite time $t^\ast(\lambda,\alpha)$ such that
\begin{equation}
		\label{eq:compressive_x1}
		p(t,x_1(t,\alpha))\to \sgn(p(0,\alpha))\infty \quad\text{as}\quad t\to t^\ast(\lambda,\alpha),
	\end{equation}
or that there is $\beta\in\mathbb{R}$ and a finite time $t^\ast(\lambda,\beta)$ such that
\begin{equation}
		\label{eq:compressive_x1}
		q(t,x_2(t,\beta))\to \sgn(q(0,\beta))\infty \quad\text{as}\quad t\to t^\ast(\lambda,\beta).
	\end{equation}
	By \eqref{eq:y_q_def}, this in turn implies that either $\partial_{x}r(t,x_1(t,\alpha))$ diverges as $t\to t^\ast(\lambda,\alpha)$ or  $\partial_{x}\ell(t,x_2(t,\beta))$ diverges as $t\to t^\ast(\lambda,\beta)$.
	Finally, according to \eqref{eq:riemann_system} and \eqref{eq:first_order_unknowns}, as a shock wave arises, the second derivatives of $w$ diverge, while its first derivatives develop a discontinuity.
\end{definition}
\begin{remark}
	\label{rmk:merging_forward_characteristics}
	In the inviscid case, a shock wave arises in conjunction with the degeneracy of characteristics, that is, when either $c_1$ or $c_2$ vanishes  \cite{paparini:singular}. This condition implies by \eqref{eq:infinitesimal_compression_ratios} that either $\partial_{x}r$ or $\partial_{x}\ell$ must diverge accordingly. However, the converse does not necessarily hold in this case.
\end{remark}
\begin{remark}
	\label{rmk:economy}
	Not to burden our presentation with too many case distinctions, we shall preferentially focus on the divergence of $p$ along forward characteristics. When both $p$ and $q$ diverge, the critical time $\tast$ is the \emph{least} between $t^\ast(\lambda,\alpha)$ and $t^\ast(\lambda,\beta)$, for all admissible $\alpha$ and $\beta$ for which these times exist.
\end{remark}
In the following section, we shall derive estimates for $t^\ast(\lambda,\alpha)$ and record without explicit proof the corresponding estimates for $t^\ast(\lambda,\beta)$.

\section{Critical time estimates}\label{sec:critical_time_estimates}
By applying the method illustrated in the preceding section, we shall determine a range of values of $\lambda>0$ for which singularities in the smooth solutions of problem \eqref{eq:wave_system} occur for a wide class of initial data $\twist_0$. We shall provide an estimate for the critical time $\tast$, which  depends only on $\lambda$ and $\twist_0$ (through $r_0$).

\subsection{Auxiliary Lemmas}
The following lemma asserts that along forward characteristic curves $x=x_1(t,\alpha)$ departing from values of $\alpha\in\mathbb{R}$ where specific properties of the initial condition $r_0$ in \eqref{eq:IC_riemann} are satisfied, the function $\eta=r-\ell$ remains confined between two functions of time determined by  $r_0$.
\begin{lemma}
\label{lemma:u_1_bounded}
Let $\lambda>0$ and let the pair $(r,\ell)$ be a solution of class $\mathcal{C}^1$ of the system \eqref{eq:diagonal_system} for $t\in[0,\tast)$ with initial condition $r_0(x)=-\ell_0(x)$ in \eqref{eq:IC_riemann}.
 Assume that $r_0$ is bounded and has a finite limit as $x\to+\infty$,
\begin{equation}
\label{eq:r_0_infty}
\lim_{x\to+\infty}r_0(x)=:r_0(+\infty)\in\mathbb{R}.
\end{equation}
If there exists at least one $\alpha\in\mathbb{R}$ such that $r_0$ satisfies both
\begin{equation}\label{eq:sufficient_initial_conditions}
\sgn\left(r_0'(\alpha)\left(r_0(\alpha)+r_0(+\infty)\right)\right)=-1\quad\text{and}\quad
\sgn\left(r_0'(x)r_0(\alpha)\right)=-1 \quad \text{for every}\ x\geq\alpha,
\end{equation}
then, along the characteristic curve $x=x_1(t,\alpha)$,  $\eta(t,x)$ is confined between two functions,
\begin{equation}
\label{eq:u_1_bar_alpha_bounded}
\begin{cases}
	\etal(t,\alpha)\leq\eta(t,x_1(t,\alpha))\leq\etau(t,\alpha)&\emph{if}\  r_0(\alpha)>0,\\
		\etau(t,\alpha)\leq\eta(t,x_1(t,\alpha))\leq\etal(t,\alpha)&\emph{if}\  r_0(\alpha)<0,
\end{cases}
\end{equation}
where
\begin{align}
\etal(t,\alpha)&:=A(t)^{-1}\left(r_0(\beta(t,\alpha))+r_0(\alpha)\right)+2 r_0(\beta(t,\alpha))\left(1-A(t)^{-1}\right),\label{eq:eta_bar_l_definitions}\\
\etau(t,\alpha)&:=A(t)^{-1}\left(r_0(\beta(t,\alpha))+r_0(\alpha)\right)+2 r_0(\alpha)\left(1-A(t)^{-1}\right).\label{eq:eta_bar_u_definitions} 
\end{align}
\end{lemma}
\begin{proof}
For the proof of this lemma, we focus on one of the alternatives contemplated in \eqref{eq:sufficient_initial_conditions}, that is,
\begin{equation}
\label{eq:sufficient_initial_condition_focus00}
r_0(\alpha)>0, \quad r_0(\alpha)+r_0(+\infty)>0\quad\text{and}\quad 
r_0'(x)<0 \quad \text{for every}\ x\geq\alpha.
\end{equation}
The other alternative can be treated in exactly the same way.
The proof employs a recursive argument to  improve progressively lower and upper bounds on the function
\begin{equation}
\label{eq:F_t}
F(t):=A(t)\eta(t,x_1(t,\alpha))=A(t)\left\{r(t,x_1(t,\alpha))-\ell(t,x_1(t,\alpha))\right\}.
\end{equation}
By Proposition~\ref{prop:beta_t_alpha}, the point $(t,x_1(t,\alpha))$, where $\ell$ is to be evaluated, is the end-point at time $t$ of the backward characteristic starting from $\beta(t,\alpha)$. Hence, by \eqref{eq:A_r_l}, 
\begin{align}
\label{eq:A_r_l_lambda}
F(t)&=A(t)\left\{r(t,x_1(t,\alpha))-\ell(t,x_2(t,\beta(t,\alpha)))\right\}\nonumber\\
&=r_0(\alpha)-\ell_0(\beta(t,\alpha))+\frac{\lambda}{2}\int_0^tA(\tau)\left\{r(\tau,x_2(\tau,\beta(t,\alpha)))-\ell(\tau,x_1(\tau,\alpha))\right\}\dd\tau.
\end{align}
This identity is instrumental to our recursive method.
\paragraph*{Step $0$.}
We start by rewriting equation \eqref{eq:A_r_l_lambda} as
\begin{equation}
\label{eq:result_step0}
F(t)=F_0(t)+\frac{\lambda}{2}\int_0^tF^{(1)}(t,\tau)\dd\tau,
\end{equation}
where
\begin{subequations}
\label{eq:ingredient_step0}
\begin{align}
F_0(t)&:=r_0(\alpha)-\ell_0(\beta(t,\alpha)),\label{eq:62_a}\\
F^{(1)}(t,\tau)&:=A(\tau)\left\{r(\tau,x_2(\tau,\beta(t,\alpha)))-\ell(\tau,x_1(\tau,\alpha))\right\}.\label{eq:62_b}
\end{align}
\end{subequations}
By Lemma~\ref{lemma:r_l_bounded} and Remark~\ref{rmk:6}, we have that 
\begin{equation}\label{eq:transition_formula}
|F^{(1)}(\tau,t)|\leq 2A(\tau)||r_0||_{\infty}, \quad \hbox{ for every } \tau\in[0,t].
\end{equation}
By using \eqref{eq:transition_formula} in \eqref{eq:result_step0}, we obtain the following upper and lower bounds for  $F(t)$,
\begin{equation}
\label{eq:A_r_l_lambda_first}
|F(t)-F_0(t)|\leq 2||r_0||_\infty\left(A(t)-1\right),
\end{equation}
where \eqref{eq:A_def} has also been used. By \eqref{eq:A_def}, 
\eqref{eq:A_r_l_lambda_first} provides a first estimate for upper and lower bounds on $F$. 
To refine this estimate we proceed recursively. We know from \eqref{eq:A_r_l} that along the characteristic curves $x_1(\cdot,\alpha)$ and $x_2(\cdot,\beta)$ through the points $(t,x)$ and $(t,y)$, 
\begin{equation}
\label{eq:A_r_l_xy}
\begin{cases}
A(t)r(t,x)=r_0(\alpha)-\frac{\lambda}{2}\int_0^{t}A(\tau)\ell(\tau,x_1(\tau,\alpha))\dd\tau,\\
A(t)\ell(t,y)=\ell_0(\beta)-\frac{\lambda}{2}\int_0^{t}A(\tau)r(\tau,x_2(\tau,\beta))\dd\tau,
\end{cases}
\end{equation}
from which it follows that
\begin{equation}
\label{eq:sum_r_lproof}
A(t)\left(r(t,x)-\ell(t,y)\right)=r_0(\alpha)-\ell_0(\beta)+\frac{\lambda}{2}\int_0^{t}A(\tau)\left\{r(\tau,x_2(\tau,\beta))-\ell(\tau,x_1(\tau,\alpha))\right\}\dd\tau.
\end{equation}
\paragraph*{Step $1$.}
We apply equation \eqref{eq:sum_r_lproof} to the integrand $F^{(1)}(t,\tau)$ in \eqref{eq:result_step0}:
$x_1(\tau,\alpha)$ and $x_2(\tau,\beta(t,\alpha))$ can be seen as end-points at time $\tau\in[0,t]$ of the characteristic curves $x_2(\cdot,\beta(\tau,\alpha))$ and $x_1(\cdot,\alpha(\tau,\beta(t,\alpha)))$, respectively, where both $\beta(\tau,\alpha)$ and $\alpha(\tau,\beta(t,\alpha))$ range between $\alpha$ and $\beta(t,\alpha)$ for  $\tau\in[0,t]$. Thus, we can write
\begin{align}
\label{eq:A_r_l_lambda_tau}
F^{(1)}(t,\tau)&=A(\tau)\left\{r(\tau,x_1(\tau,\alpha(\tau,\beta(t,\alpha))))-\ell(\tau,x_2(\tau,\beta(\tau,\alpha)))\right\}\nonumber\\
&=r_0(\alpha(\tau,\beta(t,\alpha)))-\ell_0(\beta(\tau,\alpha))+\frac{\lambda}{2}\int_0^\tau A(\xi)\left\{r(\xi,x_2(\xi,\beta(\tau,\alpha)))-\ell(\xi,x_1(\xi,\alpha(\tau,\beta(t,\alpha))))\right\}\dd\xi.
\end{align}
By use of \eqref{eq:A_r_l_lambda_tau} in \eqref{eq:result_step0}, we arrive at
\begin{equation}
\label{eq:result_step1}
F(t)=F_0(t)+\frac{\lambda}{2}\int_0^t\left[F_0^{(1)}(t,\tau)+\frac{\lambda}{2}\int_0^\tau F^{(2)}(t,\tau,\xi)\dd\xi\right]\dd\tau,
\end{equation}
where
\begin{subequations}
\label{eq:ingredient_step1}
\begin{align}
F_0^{(1)}(t,\tau)&:=r_0(\alpha(\tau,\beta(t,\alpha)))-\ell_0(\beta(\tau,\alpha)),\label{eq:ingredient_step1_1}\\
F^{(2)}(t,\tau,\xi)&:= A(\xi)\left\{r(\xi,x_2(\xi,\beta(\tau,\alpha)))-\ell(\xi,x_1(\xi,\alpha(\tau,\beta(t,\alpha))))\right\}.\label{eq:ingredient_step1_2}
\end{align}
\end{subequations}
We now estimate $F_0^{(1)}(t,\tau)$ for $\tau\in[0,t]$. By  \eqref{eq:sufficient_initial_condition_focus00}, both $\beta(\tau,\alpha)$ and $\alpha(\tau,\beta(t,\alpha))$ range between $\alpha$ and $\beta(t,\alpha)>\alpha$ for  $\tau\in[0,t]$,  and since $\ell_0=-r_0$, by \eqref{eq:ingredient_step1_1},
\begin{equation}
\label{eq:F_1_0_t}
2r_0(\beta(t,\alpha))\leq F^{(1)}_0(t,\tau)\leq 2r_0(\alpha), \quad \hbox{ for every } \tau\in[0,t].
\end{equation}
Moreover, by applying Lemma~\ref{lemma:r_l_bounded} and Remark~\ref{rmk:6} to \eqref{eq:ingredient_step1_2}, we obtain that 
\begin{equation}
\label{eq:F_1_t}
|F^{(2)}(t,\tau,\xi)|\leq 2A(\xi)||r_0||_{\infty}, \quad \hbox{ for every } \xi\in[0,\tau], \hbox{ and } \tau\in[0,t].
\end{equation}
By inserting \eqref{eq:F_1_t} into \eqref{eq:result_step1} and using again \eqref{eq:A_def}, we find that 
\begin{align}
\label{eq:A_r_l_lambda_second}
F(t)-F_0(t)&\leq 2r_0(\alpha)\frac{\lambda}{2}\int_0^t\dd\tau+ 2||r_0||_\infty	\left[\frac{\lambda}{2}\int_0^t\frac{\lambda}{2}\left(\int_0^\tau A(\xi)\dd\xi \right)\dd\tau\right]\nonumber\\
&=2r_0(\alpha)\left(\frac{\lambda t}{2}\right)+2||r_0||_{\infty}\left(-1-\frac{\lambda t}{2}+A(t)\right)
\end{align}
and
\begin{align}
\label{eq:A_r_l_lambda_second_lower}
F(t) - F_0(t)&\geq2r_0(\beta(t,\alpha))\frac{\lambda}{2}\int_0^t\dd\tau- 2||r_0||_\infty	\left[\frac{\lambda}{2}\int_0^t\frac{\lambda}{2}\left(\int_0^\tau A(\xi)\dd\xi \right)\dd\tau\right]\nonumber\\
&=2r_0(\beta(t,\alpha))\left(\frac{\lambda t}{2}\right)-2||r_0||_{\infty}\left(-1-\frac{\lambda t}{2}+A(t)\right).
\end{align}
\paragraph*{Step $2$.}
We now apply formula \eqref{eq:sum_r_lproof} to  $F^{(2)}(t,\tau,\xi)$ in \eqref{eq:ingredient_step1_2}:
$x_1(\xi,\alpha(\tau,\beta(t,\alpha)))$ and $x_2(\xi,\beta(\tau,\alpha))$ can be seen as the end-points at time $\xi$ of the characteristic curves $x_2(\cdot,\beta(\xi,\alpha(\tau,\beta(t,\alpha))))$ and $x_1(\cdot,\alpha(\xi,\beta(\tau,\alpha)))$, with $\beta(\xi,\alpha(\tau,\beta(t,\alpha))\in[\alpha(\tau,\beta(t,\alpha)),\beta(t,\alpha)]$ and $\alpha(\xi,\beta(\tau,\alpha))\in[\alpha,\beta(\tau,\alpha)]$ as $\xi\in[0,\tau]$. Therefore, for $\tau$ in $[0,t]$, both $\beta(\xi,\alpha(\tau,\beta(t,\alpha)))$ and $\alpha(\xi,\beta(\tau,\alpha))$ lie within the interval $[\alpha, \beta(t,\alpha)]$.
Then,  in \eqref{eq:result_step1} we can write 
\begin{equation}
	\label{eq:integrand_step1_to2}
	F^{(2)}(t,\tau,\xi)=F_0^{(2)}(t,\tau,\xi)+\frac{\lambda}{2}\int_0^\xi F^{(3)}(t,\tau,\xi,\omega)\dd\omega,
\end{equation}
where
\begin{subequations}
	\label{eq:ingredient_step2}
	\begin{align}
		F_0^{(2)}(t,\tau,\xi)&:=r_0(\alpha(\xi,\beta(\tau,\alpha)))-\ell_0(\beta(\xi,\alpha(\tau,\beta(t,\alpha)))),\\
		F^{(3)}(t,\tau,\xi,\omega)&:= A(\omega)\left\{r(\omega,x_2(\omega,\beta(\xi,\alpha(\tau,\beta(t,\alpha)))))-\ell(\xi,x_1(\omega,\alpha(\xi,\beta(\tau,\alpha))))\right\},
	\end{align}
\end{subequations}
and \eqref{eq:result_step1} becomes
\begin{align}
	\label{eq:result_step2}
	F(t)&=F_0(t)+\frac{\lambda}{2}\int_0^t\left\{F_0^{(1)}(t,\tau)+\frac{\lambda}{2}\int_0^\tau \left[F_0^{(2)}(t,\tau,\xi)+\frac{\lambda}{2}\int_0^\xi F^{(3)}(t,\tau,\xi,\omega)\dd\omega\right]\dd\xi\right\}\dd\tau\nonumber\\
	&=F_0(t)+\frac{\lambda}{2}\int_0^t\left[F_0^{(1)}(t,\tau)+\frac{\lambda}{2}\int_0^\tau F_0^{(2)}(t,\tau,\xi) \dd\xi\right]\dd\tau
	+\frac{\lambda}{2}\int_0^t\frac{\lambda}{2}\int_0^\tau \frac{\lambda}{2} \int_0^\xi F^{(3)}(t,\tau,\xi,\omega)\dd\omega\dd\xi\dd\tau.
\end{align}
As above, by \eqref{eq:sufficient_initial_condition_focus00}, every $F_0^{(i)}$ in \eqref{eq:result_step2} ranges between $-2r_0(\alpha)$ and $-2r_0(\beta(t,\alpha))$, while, by Lemma~\ref{lemma:r_l_bounded} and Remark~\ref{rmk:6},  $|F_1^{(2)}(\omega,\xi,\tau,t)|$ is bounded by $2||r_0||_\infty$, and so
\begin{equation}
	\label{eq:A_r_l_third_second}
	F(t)-F_0(t)\leq
	2r_0(\alpha)\left(\frac{\lambda t}{2}+\frac{(\lambda t)^2}{8}\right)+2||r_0||_{\infty}\left(-1-\frac{\lambda}{2}t-\frac{(\lambda t)^2}{8}+A(t)\right)
\end{equation}
and
\begin{equation}
	\label{eq:A_r_l_lambda_third_lower}
	F(t)-F_0(t)\geq
	2r_0(\beta(t,\alpha))\left(\frac{\lambda t}{2}+\frac{(\lambda t)^2}{8}\right)-2||r_0||_{\infty}\left(-1-\frac{\lambda t}{2}-\frac{(\lambda t)^2}{8}+A(t)\right),
\end{equation}
where use of \eqref{eq:A_def} has again been made.

The steps illustrated  above outline already a clear pattern. By  repeated applications of this method, we easily reach any higher level of recursion.

\paragraph*{Step $N$.} Thus, for $N>2$, we arrive at
\begin{subequations}
	\label{eq:step_N}
\begin{equation}
\label{eq:A_r_l_lambda_infinity}
F(t)-F_0(t)\leq2r_0(\alpha)\left[\sum_{n=0}^N\left(\frac{\left(\frac{\lambda t}{2}\right)^n}{n!}\right)-1\right]+2||r_0||_{\infty}\left[-\sum_{n=0}^N\left(\frac{\left(\frac{\lambda t}{2}\right)^n}{n!}\right)+A(t)\right]
\end{equation}
and
\begin{equation}
\label{eq:A_r_l_lambda_infinity_lower}
F(t)-F_0(t)\geq2r_0(\beta(t,\alpha))\left[\sum_{n=0}^N\left(\frac{\left(\frac{\lambda t}{2}\right)^n}{n!}\right)-1\right]-2||r_0||_{\infty}\left[-\sum_{n=0}^N\left(\frac{\left(\frac{\lambda t}{2}\right)^n}{n!}\right)+A(t)\right].
\end{equation}
\end{subequations}
Taking in \eqref{eq:step_N}  the limit as $N\to\infty$ delivers
\begin{equation}
	\label{eq:step_N=infinity}
	2r_0(\beta(t,\alpha))(A(t)-1)\leq F(t)-F_0(t)\leq2r_0(\alpha)(A(t)-1),
\end{equation}
whence the inequalities in \eqref{eq:u_1_bar_alpha_bounded} follow easily, also by use of \eqref{eq:F_t} and \eqref{eq:62_a}.
\end{proof}
\begin{remark}
Conditions \eqref{eq:sufficient_initial_conditions} are precisely the ones that in   Corollary $1$ of \cite{paparini:singular} provide an estimate for the critical time in the inviscid case. In that case, $\eta$ is constant along characteristics and, since  $\ell_0=-r_0$, both $\etal$ and $\etau$ collapse onto $\eta$,
\begin{equation}
\label{eq:eta_l_u_lambda_0}
\eta(t,x_1(t,\alpha))=r_0(\alpha)+r_0(\beta(t,\alpha)),
\end{equation}
in accord with \eqref{eq:eta_bar_l_definitions} and \eqref{eq:eta_bar_u_definitions}, as $A\equiv1$ for $\lambda=0$. 
\end{remark}
To determine the sign of  $\eta(t,x_1(t,\alpha))$ as $t$ elapses, an additional hypothesis on the sign of $r_0(+\infty)$ is required in addition to conditions \eqref{eq:sufficient_initial_conditions}.
\begin{lemma}\label{lemma:sign_eta}
Let the pair $(r,\ell)$ be a solution of class $\mathcal{C}^1$ of the system \eqref{eq:diagonal_system} for $t\in[0,\tast)$ with initial condition $(r_0,\ell_0)$ such that $r_0(x)=-\ell_0(x)$ as in \eqref{eq:IC_riemann}.  Assume that $r_0$ is bounded and has a finite limit as $x\to+\infty$, as in \eqref{eq:r_0_infty}.
Assume further that there is at least one $\alpha\in\mathbb{R}$ such that, in addition to \eqref{eq:sufficient_initial_conditions}, $r_0$ satisfies
\begin{equation}
\label{eq:sufficient_initial_condition3_sec0}
\sgn(r_0(\alpha)r_0(+\infty))=+1 \quad \hbox{or} \quad r_0(+\infty)=0.
\end{equation}
Then,
\label{eq:u_1_alpha_bounded_sign}
\begin{equation}
\sgn(\eta(t,x_1(t,\alpha)\eta(0,\alpha))=+1\label{eq:sign_eta_x1}
\end{equation}
along the characteristic curve $x=x_1(t,\alpha)$.
\end{lemma}
\begin{proof}
The functions $\etal(t,\alpha)$ in \eqref{eq:eta_bar_l_definitions} and $\etau(t,\alpha)$ in \eqref{eq:eta_bar_u_definitions} are both different from zero for every $t\in[0,\tast)$ provided that, in addition to \eqref{eq:sufficient_initial_conditions}, $r_0$ also satisfies \eqref{eq:sufficient_initial_condition3_sec0}. Consequently, they keep the same sign as $\etal(0,\alpha)=\etau(0,\alpha)=2r_0(\alpha)$ for every $t\in[0,\tast)$. 
\end{proof}
\begin{remark}
\label{rmk:eta_l_decreasing}
We note that, under conditions \eqref{eq:sufficient_initial_conditions}, the function $\etal(t,\alpha)$ in \eqref{eq:eta_bar_l_definitions} reproduces for $t>0$ the same behaviour $r_0(x)$ has for $x\geq\alpha$. Indeed, since 
\begin{equation}
\label{eq:eta_l_decreasing}
\partial_t\etal(t,\alpha)=-\frac{\lambda}{2}A(t)^{-1}(r_0(\alpha)-r_0(\beta(t,\alpha)))+r_0'(\beta(t,\alpha))\partial_{t}\beta(t,\alpha)(2-A(t)^{-1}),
\end{equation}
by \eqref{eq:d_t_beta} we see that 
\begin{equation}
\sgn(\partial_t\etal(t,\alpha)r_0'(\alpha))=+1 \quad \text{for every } t\in[0,\tast).
\end{equation}
\end{remark}
In the following lemma, we state a property similar to that found in Remark~\ref{rmk:eta_l_decreasing} for the function  $\etal(t,\alpha)$.
\begin{lemma}
\label{lemma:int_q_sign}
Let the pair $(r,\ell)$ be a solution of class $\mathcal{C}^1$ of the system \eqref{eq:diagonal_system} for $t\in[0,\tast)$ with initial condition $(r_0,\ell_0)$ such that $r_0(x)=-\ell_0(x)$ as in \eqref{eq:IC_riemann}. Assume that $r_0$ is bounded and has a finite limit as $x\to+\infty$, as in \eqref{eq:r_0_infty}.
Assume further  that there exists at least one $\alpha\in\mathbb{R}$ such that \eqref{eq:sufficient_initial_conditions}, and \eqref{eq:sufficient_initial_condition3_sec0} hold. Then, for every $t\in[0,\tast)$ such that
\begin{equation}
\label{integral_q_positive}
-\frac{2\left(A(t)-1\right)}{\sqrt{k(\etal(t,\alpha))}}+\frac{2-A(t)^{-1}}{\sqrt{k(\eta(0,\alpha))}}\geq 0,
\end{equation}
we have that
\begin{equation}
\label{eq:integral_q}
\sgn\left(r_0'(\alpha)\left(\int_0^tq(\tau,x_1(\tau,\alpha))\dd\tau\right)\right)=-1,
\end{equation}
where $q$ is defined as in \eqref{eq:q_def}. 
\end{lemma}
\begin{proof}
Here, we focus on the specific case of \eqref{eq:sufficient_initial_conditions} and \eqref{eq:sufficient_initial_condition3_sec0} embodied by 
\begin{equation}\label{eq:sufficient_initial_condition_focus}
r_0(\alpha)>0, \quad r_0(+\infty)\geq0 \quad\text{and}\quad
r_0'(x)<0 \quad \text{for every} \, x\geq\alpha.
\end{equation}
The complementary case arising from assuming that $r_0(\alpha)<0$  is analogous and can be treated similarly.
By \eqref{eq:x_1}, \eqref{eq:partial_plus}, and \eqref{eq:partial_plus_minus_u1}, we can rewrite the definition of $q$ in \eqref{eq:q_def} as
\begin{equation}
\label{eq:q_x1}
q(t,x_1(t,\alpha))=-\frac{1}{2}\frac{A(t)}{\sqrt{k(\eta(t,x_1(t,\alpha)))}}\partial^+\eta=-\frac12A(t)\frac{\dd}{\dd t}H(\eta(t,x_1(t,\alpha))),
\end{equation}
where $H$ is  such that
\begin{equation}
\label{eq:H_map}
H'(\eta)=\frac{1}{\sqrt{k(\eta)}}>0.
\end{equation}
By integrating by parts in \eqref{eq:q_x1}, we obtain that
\begin{equation}
\label{eq:integral_q_byparts}
\int_0^tq(\tau,x_1(\tau,\alpha))\dd\tau=\frac{1}{2}\left[-A(t)H(\eta(t,x_1(t,\alpha)))+H(\eta(0,\alpha))+\frac{\lambda}{2}\int_0^tA(\tau)H(\eta(\tau,x_1(\tau,\alpha)))\dd\tau\right],
\end{equation}
and to prove \eqref{eq:integral_q} it suffices to show that  the right-hand side of \eqref{eq:integral_q_byparts} is positive. Under hypothesis \eqref{eq:sufficient_initial_condition_focus}, by Lemma~\ref{lemma:sign_eta} and Remark~\ref{rmk:eta_l_decreasing}, the function $\eta(t,x_1(t,\alpha))$ is always positive and bounded below by the positive and strictly decreasing function $\etal(t,\alpha)$ in \eqref{eq:eta_bar_l_definitions}. Accordingly, by \eqref{eq:H_map} $H(\etal(t,\alpha))$ is a decreasing function of $t$ for every $t\in[0,\tast)$, and so 
\begin{equation}
\label{eq:H_inside_integral_1}
H(\eta(\tau,x_1(\tau,\alpha)))\geq H(\etal(\tau,\alpha))\geq H(\etal(t,\alpha)).
\end{equation}
Moreover, since 
\begin{equation}
	\label{eq:H_double_prime}
H''(\eta)=-\frac{f(\eta)}{2k(\eta)}\quad\text{and}\quad f(\eta)\eta\geq0,
\end{equation}
$H$ is  concave for every $\eta\geq0$. Thus, the right-hand side of \eqref{eq:integral_q_byparts} can  be estimated as
\begin{subequations}
\label{integral_q_byparts_estimated0}
\begin{align}
\int_0^tq(\tau,x_1(\tau,\alpha))&\geq\frac12\left[ A(t)\left(H(\etal(t,\alpha))-H(\etau(t,\alpha))\right)+H(\eta(0,\alpha))-H(\etal(t,\alpha))\right]\label{eq:ineq1}\\
&\geq\frac12 \left[A(t)H'(\etal(t,\alpha))\left(\etal(t,\alpha)-\etau(t,\alpha)\right)+H'(\eta(0,\alpha))\left(\eta(0,\alpha)-\etal(t,\alpha)\right)\right]\label{eq:ineq2}\\
&=\frac12(r_0(\alpha)-r_0(\beta(t,\alpha)))\left[-\frac{2\left(A(t)-1\right)}{\sqrt{k(\etal(t,\alpha))}}+\frac{2-A(t)^{-1}}{\sqrt{k(\eta(0,\alpha))}}\right]\label{eq:ineq3},
\end{align}
\end{subequations}
where inequality \eqref{eq:ineq2} follows from  the concavity of  $H$ and use is made in \eqref{eq:ineq3} of \eqref{eq:eta_bar_l_definitions}, \eqref{eq:eta_bar_u_definitions}, and \eqref{eq:H_map}. By \eqref{eq:sufficient_initial_condition_focus}, $r_0(\alpha)-r_0(\beta(t,\alpha))>0$ for $t>0$, and so \eqref{eq:ineq3}implies \eqref{eq:integral_q}.
\end{proof}
\begin{remark}
\label{rmk:estimate_approx_delta}
The function $\beta(t,\alpha)$ is hard to find explicitly, but its bounds  in \eqref{eq:beta_confined} provide explicit estimates for $\etal$:
\begin{equation}
\label{eq:lower_bound_delta}
\begin{cases}
\etal(t,\alpha)\geq\etaLp(t,\alpha)\geq0&\textrm{if}\ r_0(\alpha)>0,\\
\etal(t,\alpha)\leq\etaLp(t,\alpha)\leq0&\textrm{if}\ r_0(\alpha)<0,
\end{cases}
\end{equation}
where 
\begin{equation}
\label{eq:lower_bound_eta_delta}
\etaLp(t,\alpha):=A(t)^{-1}\left(r_0(\alpha+2\delta t)+r_0(\alpha)\right)+2 r_0(\alpha+2\delta t)\left(1-A(t)^{-1}\right),
\end{equation}
with $\delta$ defined in \eqref{eq:estimates}. 
Since the function $\alpha+2\delta t$ is monotonically increasing with time, the function $\etaLp(t,\alpha)$ retains the same properties as $\etal(t,\alpha)$ outlined in Lemma~\ref{lemma:sign_eta} and Remark~\ref{rmk:eta_l_decreasing}: thus,   under conditions \eqref{eq:sufficient_initial_conditions} and \eqref{eq:sufficient_initial_condition3_sec0}, 
\begin{equation}
\label{eq:properties_eta_l_delta}
\sgn(\partial_{t}\etaLp(t,\alpha)r_0'(\alpha))=+1 \quad\text{and}\quad \sgn(\partial_{t}\etaLp(t,\alpha)\eta(0,\alpha))=+1 \quad \hbox{ for every } t\in[0,\tast).
\end{equation}
Moreover, from \eqref{eq:lower_bound_delta} we obtain that
\begin{equation}
\label{integral_q_positive_approx0}
-\frac{2\left(A(t)-1\right)}{\sqrt{k(\etal(t,\alpha))}}+\frac{2-A(t)^{-1}}{\sqrt{k(\eta(0,\alpha))}}\geq -\frac{2\left(A(t)-1\right)}{\sqrt{k(\etaLp(t,\alpha))}}+\frac{2-A(t)^{-1}}{\sqrt{k(\eta(0,\alpha))}},
\end{equation}
and so the conclusion of Lemma~\ref{lemma:int_q_sign} also applies under the following condition
\begin{equation}
\label{integral_q_positive_approx}
-\frac{2\left(A(t)-1\right)}{\sqrt{k(\etaLp(t,\alpha))}}+\frac{2-A(t)^{-1}}{\sqrt{k(\eta(0,\alpha))}}\geq0,
\end{equation}
which strengthens \eqref{integral_q_positive} and will replace it in our development below.
\end{remark}

\subsection{Main Result}
We now provide an estimate for the critical time $t^\ast(\lambda,\alpha)$  at which a singularity occurs in a forward characteristics under the assumptions  \eqref{eq:sufficient_initial_conditions} and \eqref{eq:sufficient_initial_condition3_sec0}. This estimate will be acceptable only  if it satisfies \eqref{integral_q_positive_approx}.
\begin{theorem}
\label{th:suff_cond}
Consider the global Cauchy problem \eqref{eq:diagonal_system} with initial condition $(r_0,\ell_0)$ such that $r_0(x)=-\ell_0(x)\in \mathcal{C}^1$ as in \eqref{eq:IC_riemann}. 
Suppose that for a given $\lambda>0$ there exists at least one $\alpha\in\mathbb{R}$ such that $r_0$ satisfies conditions \eqref{eq:sufficient_initial_conditions} and \eqref{eq:sufficient_initial_condition3_sec0}. Then the solution $r(t,x)$ to \eqref{eq:diagonal_system} is predicted to develop a singularity along the characteristic curve $x=x_1(t,\alpha)$ in a finite time $t^*(\lambda,\alpha)\leq\tc(\lambda,\alpha)$, provided that there exists a solution to the following equation for  $\tc$,
\begin{equation}
\label{eq:critical_time}
1-\frac{2}{\lambda}\left(1-A(\tc)^{-1}\right)\gamma^+(\tc;\lambda,\alpha)|r'_0(\alpha)|\sqrt{k(2r_0(\alpha))}=0,
\end{equation}
where $\gamma^+$ is defined as
\begin{equation}
\label{eq:gamma_alpha}
\gamma^+(t;\lambda,\alpha):=\min\left\{|f(\etaLp(t,\alpha))|, |f(2r_0(\alpha))|\right\},
\end{equation}
with $\etaLp$ as in \eqref{eq:lower_bound_eta_delta}, and provided that at $t=\tc(\lambda,\alpha)$  inequality \eqref{integral_q_positive_approx} is satisfied.
\end{theorem}
\begin{proof}
As above, for the proof of this theorem, we focus again  on the specific case where conditions \eqref{eq:sufficient_initial_conditions} and \eqref{eq:sufficient_initial_condition3_sec0} are specialized as in \eqref{eq:sufficient_initial_condition_focus}. 
We first consider the ordinary differential equation in \eqref{eq:partial_plus_y}. By integrating in time, $p$ can be written along the forward characteristic $x_1(t,\alpha)$ as
\begin{equation}
\label{eq:y_along_characteristic}
p(t,x_1(t,\alpha))=p(0,\alpha)-\frac{\lambda}{2}\int_0^tq(\tau,x_1(\tau,\alpha))\dd\tau-\int_0^tA(\tau)^{-1}f(\eta(\tau,x_1(\tau,\alpha)))p(\tau,x_1(\tau,\alpha))^2\dd\tau.
\end{equation}
Here, by \eqref{eq:y_def},
\begin{equation}\label{eq:p_initial}
	p(0,\alpha)=\sqrt{k(2r_0(\alpha))}r_0'(\alpha)<0.
\end{equation}
By Lemma~\ref{lemma:int_q_sign}, for every $t$ such that \eqref{integral_q_positive_approx} holds, 
\begin{equation}\label{eq:positive_integral}
\int_0^tq(\tau,x_1(\tau,\alpha))\dd\tau>0,
\end{equation}
and so 
\begin{equation}
\label{eq:y_along_characteristic_ineq}
p(t,x_1(t,\alpha))\leq p(0,\alpha)-\int_0^tA(\tau)^{-1}f(\eta(\tau,x_1(\tau,\alpha)))p(\tau,x_1(\tau,\alpha))^2\dd\tau.
\end{equation}
Moreover, according to Lemma~\ref{lemma:sign_eta}, under conditions \eqref{eq:sufficient_initial_condition_focus} $\eta(t,x_1(t,\alpha))$ is  positive for every $t\in[0,t^*(\lambda,\alpha))$. Since $\sgn(f(\eta)\eta)=+1$ for all $\eta\neq0$,  $f(\eta(\tau,x_1(\tau,\alpha)))$ in \eqref{eq:y_along_characteristic_ineq} is also positive for every $\tau\in[0,t]$, and so
\begin{equation}
\label{eq:Psi_def}
p(t,x_1(t,\alpha))\leq\Psi(t):=p(0,\alpha)-\int_0^tA(\tau)^{-1}f(\eta(\tau,x_1(\tau,\alpha)))p(\tau,x_1(\tau,\alpha))^2\dd\tau<0,
\end{equation}
from which we arrive at
\begin{equation}
\label{eq:dot_Psi}
\dot\Psi(t)=-A(t)^{-1}f(\eta(t,x_1(t,\alpha)))p(t,x_1(t,\alpha))^2\leq-A(t)^{-1}f(\eta(t,x_1(t,\alpha)))\Psi(t)^2.
\end{equation}
Let now $\Phi(t)$ satisfy
\begin{equation}
\label{eq:zeta_ode}
\dot{\Phi}(t)+A(t)^{-1}f(\eta(t,x_1(t,\alpha)))\Phi(t)^2=0, \quad \Phi(0)=p(0,\alpha)<0,
\end{equation}
so that, by \eqref{eq:dot_Psi},
\begin{equation}
\label{eq:comparison_v_zeta}
\dot{\Psi}(t)+A(t)^{-1}f(\eta(t,x_1(t,\alpha)))\Psi(t)^2\leq0=\dot{\Phi}(t)+A(t)^{-1}f(\eta(t,x_1(t,\alpha)))\Phi(t)^2.
\end{equation}
By applying to \eqref{eq:comparison_v_zeta} the comparison theorem stated in Appendix~\ref{sec:aux_res}, by \eqref{eq:Psi_def}, we find that 
\begin{equation}
\label{eq:v_leq_zeta}
p(t,x_1(t,\alpha))\leq\Psi(t)\leq\Phi(t), \quad \text{for every } t\in[0,t^*(\lambda,\alpha)),
\end{equation}
where $\Phi(t)$ is the solution of \eqref{eq:zeta_ode},
\begin{equation}
\label{eq:solution_zeta}
\Phi(t)=
-|p(0,\alpha)|\left(1-|p(0,\alpha)|\int_0^tA(\tau)^{-1}f(\eta(\tau,x_1(\tau,\alpha)))\dd\tau\right)^{-1}<0.
\end{equation}
By \eqref{eq:p_initial}, $\Phi$ diverges to $-\infty$ in a finite time $t^\ast_\mathrm{c}(\lambda,\alpha)$ if
\begin{equation}
\label{eq:denominator_critical}
g(t^\ast_\mathrm{c};\lambda,\alpha):=1-\sqrt{k(2r_0(\alpha))}|r'_0(\alpha)|\int_0^{t^\ast_\mathrm{c}}A(\tau)^{-1}f(\eta(\tau,x_1(\tau,\alpha)))\dd\tau=0
\end{equation}
and, by \eqref{eq:v_leq_zeta} and \eqref{eq:y_def}, so also does  $\partial_{x}r(t,x)$ along the characteristic curve $x=x_1(t,\alpha)$.

To arrive at \eqref{eq:critical_time} from \eqref{eq:denominator_critical}, we observe that,  by Lemma~\ref{lemma:sign_eta}, Remark~\ref{rmk:estimate_approx_delta}, and \eqref{eq:properties_eta_l_delta}, 
\begin{equation}
\label{eq:H_inside_integral_2}
\etaLp(t,\alpha)\leq\etaLp(\tau,\alpha))\leq \eta(\tau,x_1(\tau,\alpha))\leq 2r_0(\alpha)\quad\text{for every }\tau\in[0,t],
\end{equation}
and so
\begin{align}
\label{eq:denominator_critical_plus}
g(t;\lambda,\alpha)&\leq1-\sqrt{k(2r_0(\alpha))}|r'_0(\alpha)|\gamma^+(t;\lambda,\alpha)\int_0^{t}A(\tau)^{-1}\dd\tau\nonumber\\
&=1-\sqrt{k(2r_0(\alpha))}|r'_0(\alpha)|\gamma^+(t;\lambda,\alpha)\frac{2}{\lambda}\left(1-A(t)^{-1}\right) \quad\text{for every }t>0,
\end{align}
where $\gamma^+(t;\lambda,\alpha)$ is defined by \eqref{eq:gamma_alpha}. Thus, letting $t_\mathrm{c}(\lambda,\alpha)$ be a root of \eqref{eq:critical_time}, provided that it exists and satisfies \eqref{eq:y_along_characteristic_ineq}, \eqref{eq:denominator_critical_plus} implies that
\begin{equation}
	\label{eq:time:hierarchy}
	t^\ast(\lambda,\alpha)\leq t^\ast_\mathrm{c}(\lambda,\alpha)\leq t_\mathrm{c}(\lambda,\alpha).
\end{equation}
\end{proof}
\begin{remark}\label{rmk:ODE}
The proof of Theorem~\ref{th:suff_cond} essentially relies  on the existence for all times of solutions for the following Riccati equation,
\begin{equation}
\label{eq:riccati_equation}
\dot v(t)+G(t)v(t)^2=b(t),
\end{equation}
which reduces to \eqref{eq:partial_plus_y} once we set
\begin{subequations}
\label{eq:ingredients_riccati}
\begin{align}
v(t)&:=p(t,x_1(t,\alpha)),\\
G(t)&:=A(t)^{-1}f(\eta(t,x_1(t,\alpha))),\\
b(t)&:=-\frac{\lambda}{2}A(t)\sqrt{k(\eta(t,x_1(t,\alpha)))}\partial_{x}\ell(t,x_1(t,\alpha)).
\end{align}
\end{subequations}
Were we able to determine the sign of $b(t)$ or verify whether $b(t)\leq s(t)$, for some differentiable and non-decreasing function $s(t)$, we could either guarantee   solvability of \eqref{eq:riccati_equation} for all times \cite{grigoryan:global}, or  blow-up in a finite time \cite{wei:new,baris:blowup,dou:analytical}. However, by \eqref{eq:partial_plus_minus_u1}, controlling the sign of $b$ amounts to control the sign of the derivative $\partial^+\eta$ along forward characteristics, which is a far more demanding task than the bounds delivered by \eqref{eq:u_1_bar_alpha_bounded} can accomplish. In our study, the control of the sign of $b$ is replaced by the estimate in \eqref{eq:y_along_characteristic_ineq}.
\end{remark}
\begin{remark}
A parallel argument applied to backward characteristics proves that if $r_0$ has a finite limit as $x\to-\infty$,
\begin{equation}
\label{eq:r_0_-infty}
\lim_{x\to-\infty}r_0(x)=:r_0(-\infty)\in\mathbb{R},
\end{equation}
and if there exists at least one $\beta\in\mathbb{R}$ such that $r_0(x)$ satisfies the conditions
\begin{subequations}
\begin{gather}
\label{eq:sufficient_condition_minus_infinity}
\sgn\left(r_0'(\beta)\left(r_0(\beta)+r_0(-\infty)\right)\right)=+1\quad\text{and}\quad
\sgn\left(r_0'(x)r_0(\beta)\right)=+1, \quad \text{for every } x\leq\beta,\\
\sgn\left(r_0(\beta)r_0(-\infty)\right)=+1,
\end{gather}
\end{subequations}
then the solution $\ell(t,x)$ to \eqref{eq:diagonal_system} will develop a singularity along the characteristic curve $x=x_2(t,\beta)$ in a finite time $t^\ast(\lambda,\beta)\leq\tc(\lambda,\beta)$, provided that 
\begin{enumerate}
	\item there exists a solution to the following equation for  $\tc$,
\begin{equation}
\label{eq:critical_time_beta}
1-\frac{2}{\lambda}\left(1-A(\tc)^{-1}\right)\gamma^-(\tc;\lambda,\beta)|r'_0(\beta)|\sqrt{k(2r_0(\beta))}=0,
\end{equation}
where $\gamma^-$ is defined as
\begin{equation}
\label{eq:gamma_beta}
\gamma^-(t;\lambda,\beta):=\min\left\{|f(\etaLm(t,\beta))|, |f(2r_0(\beta))|\right\}
\end{equation}
and
\begin{equation}
\label{eq:lower_bound_delta_beta}
\etaLm(t,\beta):=A(t)^{-1}\left(r_0(\beta-2\delta t)+r_0(\beta)\right)+2 r_0(\beta-2\delta t)\left(1-A(t)^{-1}\right),
\end{equation}
\item 
at $t=\tc(\lambda,\beta)$ the following inequality is satisfied,
\begin{equation}
\label{integral_q_positive_approx_beta}
-\frac{2\left(A(t)-1\right)}{\sqrt{k(\etaLm(t,\beta))}}+\frac{2-A(t)^{-1}}{\sqrt{k(\eta(0,\beta))}}\geq0.
\end{equation}
\end{enumerate}
\end{remark}
\begin{remark}
We now consider the inviscid limit of Theorem~\ref{th:suff_cond}. For $\lambda=0$,  \eqref{eq:eta_l_u_lambda_0} applies and since, by \eqref{eq:beta_confined}, $\beta(t,\alpha)\geq\alpha$ for every $t\geq0$, \eqref{eq:d_t_beta} implies that under conditions \eqref{eq:sufficient_initial_conditions} $r_0(\alpha)+r_0(\beta(t,\alpha))$ is monotonic and does not change sign for all $t\geq0$. Thus, the asymptotic limit $r_0(\alpha)+r_0(+\infty)$ approached as $t\to\infty$ has the same sign as $r_0(\alpha)+r_0(\beta(0,\alpha))$. This implies that
\begin{equation}
|\etaLp(t,\alpha)|=|r_0(\alpha)+r_0(\alpha+2\delta t)|\geq|r_0(\alpha)+r_0(+\infty)|\quad\text{for every }t\geq0,
\end{equation}
so that 
\begin{equation}
\label{eq:lambda_to_zero_critical_time}
\gamma^+(t;0,\alpha):=	\lim_{\lambda\to0}\gamma^+(t;\lambda,\alpha)=
\min\left\{\left|f(2r_0(\alpha))\right|,\left|f(r_0(\alpha)+r_0(+\infty))\right|\right\}\quad\text{for every }t\geq0
\end{equation}
and \eqref{eq:critical_time} becomes
\begin{equation}
	\label{eq:critical_time_inviscid}
	\tc(\alpha,0)=\frac{1}{\sqrt{k(2r_0(\alpha))}|r_0'(\alpha)|\gamma^+(t;0,\alpha)},
\end{equation}
which is precisely the estimate achieved in Corollary~1 of \cite{paparini:singular}.
 We finally also note that in the inviscid case,  inequality \eqref{integral_q_positive_approx} need not be satisfied.
\end{remark}
Theorem~\ref{th:suff_cond} estimates $t^\ast(\lambda,\alpha)$, for given $\lambda$ and $\alpha$. An estimate for the critical time $\tast$ follows from minimizing $\tc(\lambda,\alpha)$ over all $\alpha$'s for which a singularity occurs, as summarized in the following proposition. \begin{proposition}
Let the hypotheses of Theorem~\ref{th:suff_cond} be satisfied and let 
\begin{equation}
	\label{eq:admissible_set_definition}
	\adm:=\{\alpha\in\mathbb{R}:\exists\ \tc\ \text{that solves \eqref{eq:critical_time} subject to \eqref{integral_q_positive_approx}}\}.
\end{equation}
Then, for given $\lambda>0$, the critical time $\tast$ for the existence of a solution of class $\mathcal{C}^1$ to system \eqref{eq:diagonal_system}  can be estimated as 
\begin{equation}
\label{eq:t_ast_definition}
\tast\leq \tc(\lambda):=\begin{cases}
	\inf_{\alpha\in\adm}\tc(\lambda,\alpha)&\emph{if}\ \adm\neq\emptyset,\\
	 +\infty &\emph{if}\ \adm=\emptyset.
\end{cases}
\end{equation}
\end{proposition}
\begin{remark}
	It should be noted that whenever $\adm=\emptyset$, \eqref{eq:t_ast_definition} does \emph{not} imply that the solution to system \eqref{eq:diagonal_system} is smooth for all times: more simply, our sufficient criterion for criticality does not apply.
\end{remark}

\section{Application}\label{sec:applications}
As an illustrative example, we consider an initial profile $w_0$ for the twist angle that exhibits a strong concentration of distortion around $x=0$, which fades away at infinity without ever vanishing. Thus, by \eqref{eq:kernel}, the more distorted core is expected to propagate faster than the distant tails, possibly overtaking them: intuitively, this should prompt the creation of a singularity in a finite time. We shall see here how such an intuitive prediction is indeed confirmed by the estimate \eqref{eq:t_ast_definition}.

We consider the following  initial twist  profile
\begin{equation}
\label{eq:w_0_arctan}
\twist_0(\kappa,\zeta;x):=-\frac{2\kappa}{\pi}\arctan\frac{x}{\zeta},
\end{equation}
where $\kappa$ and $\zeta$ are positive parameters. This is a \emph{kink} representing a smooth transition of the twist angle between two asymptotic values depending on $\kappa$ over an effective  width depending on $\zeta$.
Fig.~\ref{fig:w_0arctan} illustrates the graphs of the initial profile $\twist_0$ in \eqref{eq:w_0_arctan} for different values of $\zeta$ and $\kappa$: either increasing  $\zeta$ or decreasing  $\kappa$ makes the initial profile \emph{less} distorted (whereas either decreasing $\zeta$ or increasing $\kappa$ makes the initial profile \emph{more} distorted).
\begin{figure}[]
	\centering
	\begin{subfigure}[c]{0.35\linewidth}
		\centering
		\includegraphics[width=0.9\linewidth]{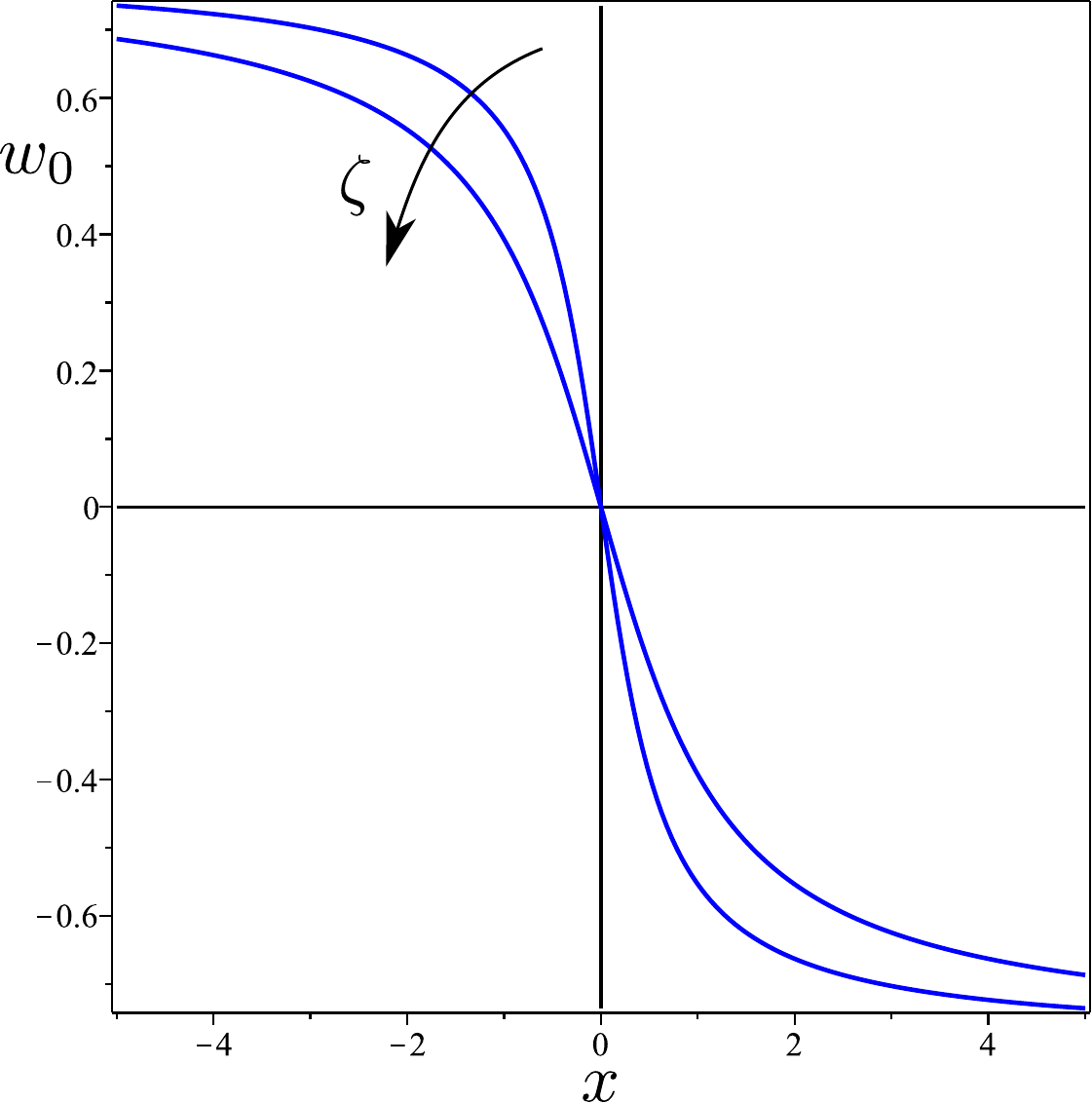}
		\caption{$\kappa=\pi/4$, $\zeta=1/2, \, 1$.} 
		\label{fig:w_0arctanzeta}
	\end{subfigure}
	\begin{subfigure}[c]{0.35\linewidth}
		\centering
		\includegraphics[width=0.9\linewidth]{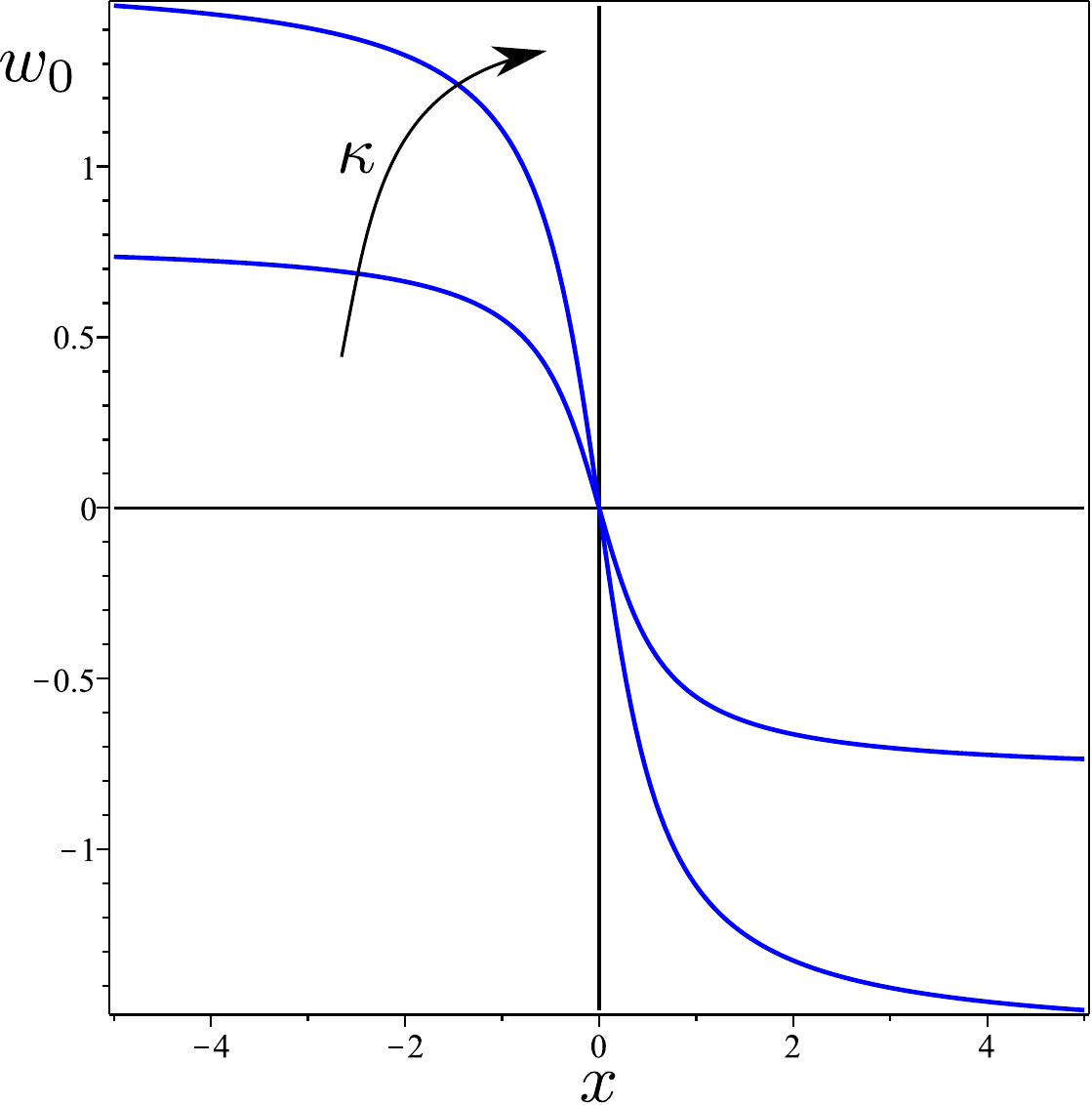}
		\caption{$\zeta=1/2$, $\kappa=\pi/4, \, \pi/2$.} 
		\label{fig:w_0arctankappa}
	\end{subfigure}
\caption{Initial profile of the twist angle $w_0$ represented by  \eqref{eq:w_0_arctan} for several values of the parameters $\kappa$ and $\zeta$, which describe the amplitude and width of the kink, respectively.}
	\label{fig:w_0arctan}
\end{figure}

By \eqref{eq:IC_riemann} and \eqref{eq:L_u_1}, $r_0$ is given by
\begin{equation}
\label{eq:r_0arctan}
r_0(\kappa,\zeta;x)=\frac{1}{2}\left(\xi\sqrt{1+\xi^2}+\hbox{arcsinh}\,\xi\right), \quad \, \hbox{where} \quad \xi:=\frac{2\kappa\zeta}{\pi(x^2+\zeta^2)}.
\end{equation}
The  graph of $r_0$ is illustrated in Fig.~\ref{fig:r_0arctan} for different values of $\kappa$ and $\zeta$. Since $r_0$ is an even function,  by Theorem~\ref{th:suff_cond},  we  need only  consider forward characteristics: if a singularity arises along a forward characteristic originating at $\alpha$, a singularity will also occur at the same critical time along the (symmetric) backward characteristic originating at $\beta=-\alpha$.
\begin{figure}[]
	\centering
	\begin{subfigure}[c]{0.35\linewidth}
		\centering
		\includegraphics[width=0.9\linewidth]{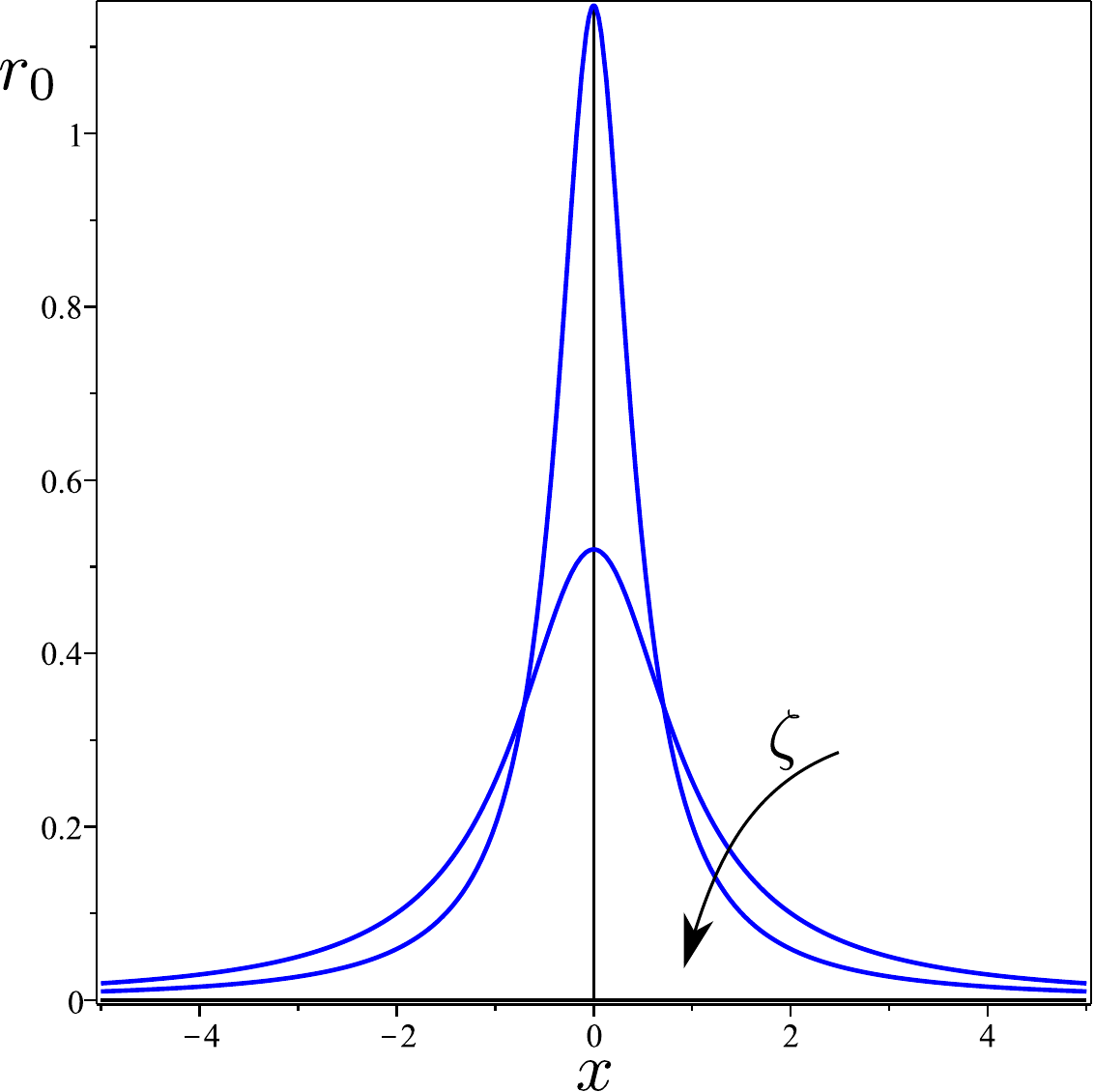}
		\caption{$\kappa=\pi/4$, $\zeta=1/2, \, 1$.} 
		\label{fig:w_0arctanzeta}
	\end{subfigure}
	\begin{subfigure}[c]{0.35\linewidth}
		\centering
		\includegraphics[width=0.9\linewidth]{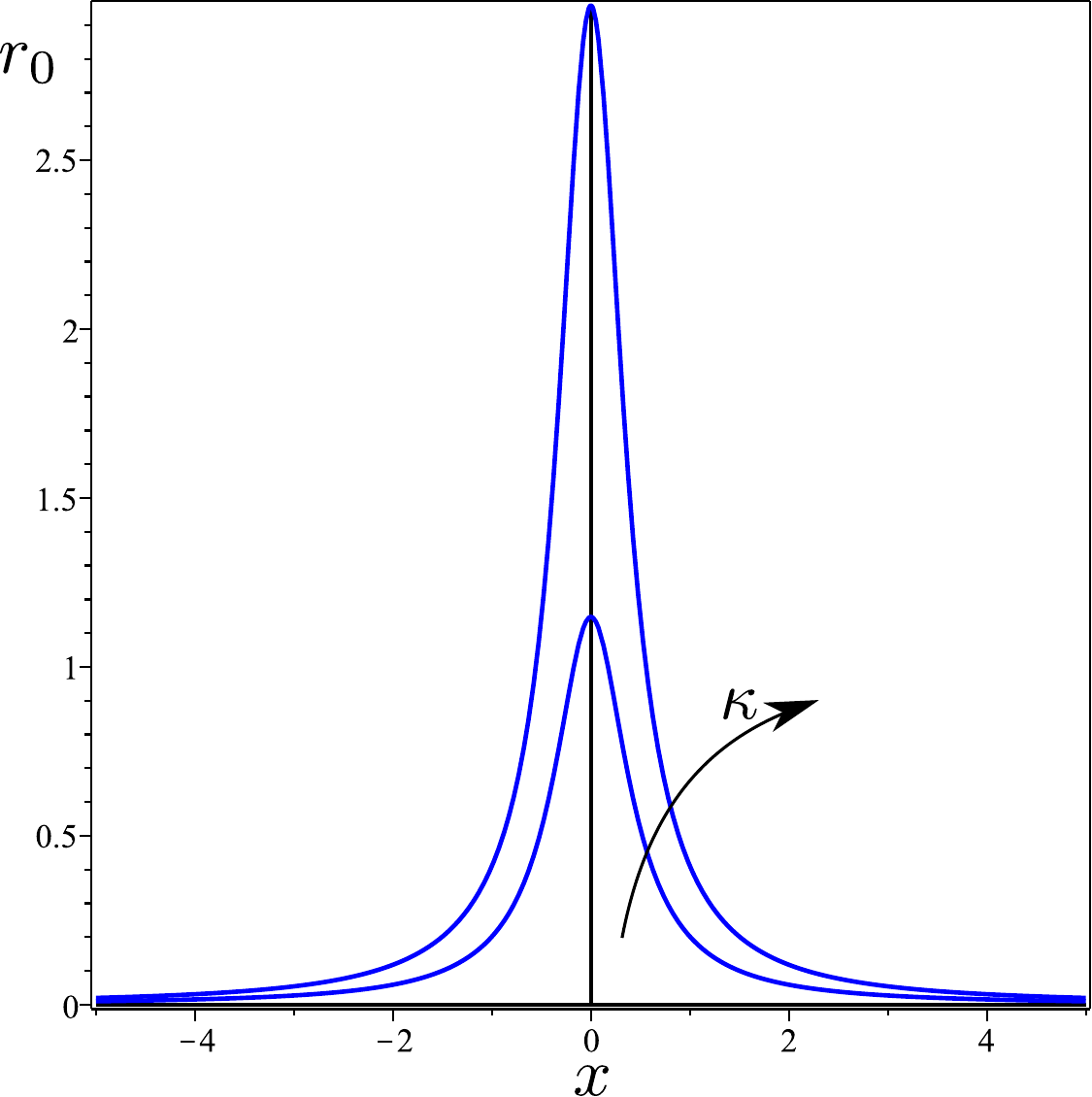}
		\caption{$\zeta=1/2$, $\kappa=\pi/4, \, \pi/2$.} 
		\label{fig:w_0arctankappa}
	\end{subfigure}
\caption{Initial profile of the Riemann function $r_0$ in \eqref{eq:r_0arctan} corresponding to  the initial twist $w_0$ illustrated in Fig.~\ref{fig:w_0arctan}.}
	\label{fig:r_0arctan}
\end{figure}

Here, $r_0(+\infty)=0$ and conditions \eqref{eq:sufficient_initial_conditions} and \eqref{eq:sufficient_initial_condition3_sec0} are satisfied for every $\alpha>0$. Thus, by Theorem~\ref{th:suff_cond} a singularity can occur in a finite time $t^\ast(\lambda,\alpha)\leq \tc(\lambda,\alpha)$, provided that there is a root of \eqref{eq:critical_time} for $\tc$ that obeys \eqref{integral_q_positive_approx}.

To compute the infimum of $\adm$ for any given $\lambda>0$, we find it convenient to  compute first the \emph{pre-critical} time $\widehat{t}_\mathrm{c}(\lambda)$ defined as
\begin{equation}
	\label{eq:pre_critical_time}
	\widehat{t}_\mathrm{c}(\lambda):=\inf_{\alpha\in\admh}\tc(\lambda,\alpha),
\end{equation}
where
\begin{equation}
	\label{eq:pre_admissible_set}
	\admh:=\{\alpha\in\mathbb{R}:\exists\ \tc\ \text{that solves \eqref{eq:critical_time}}\},
\end{equation}
and then use \eqref{integral_q_positive_approx} as a selection rule: we set $\tc(\lambda)=\widehat{t}_\mathrm{c}(\lambda)$, if \eqref{integral_q_positive_approx} is valid for $t=\widehat{t}_\mathrm{c}(\lambda)$ and $\alpha=\arginf\tc(\lambda,\alpha)$;\footnote{In most practical cases, the infimum is attained and \eqref{integral_q_positive_approx} is computed for $\alpha=\argmin\tc(\lambda,\alpha)$.} otherwise, we set $\tc(\lambda)=+\infty$. 

We start by studying the function $\etaLp(t,\alpha)$ in \eqref{eq:lower_bound_eta_delta}; its graph against $t>0$  is represented in Fig.~\ref{fig:eta_l_delta_graph} for different values of $\lambda>0$. As observed in Lemma~\ref{lemma:sign_eta} and Remark~\ref{rmk:eta_l_decreasing}, its behavior reproduces for $t>0$ that of $r_0(x)$ for $x\geq\alpha$: it is strictly decreasing and stays positive as time elapses.

\begin{figure}[] 
	\centering
	\includegraphics[width=0.45\linewidth]{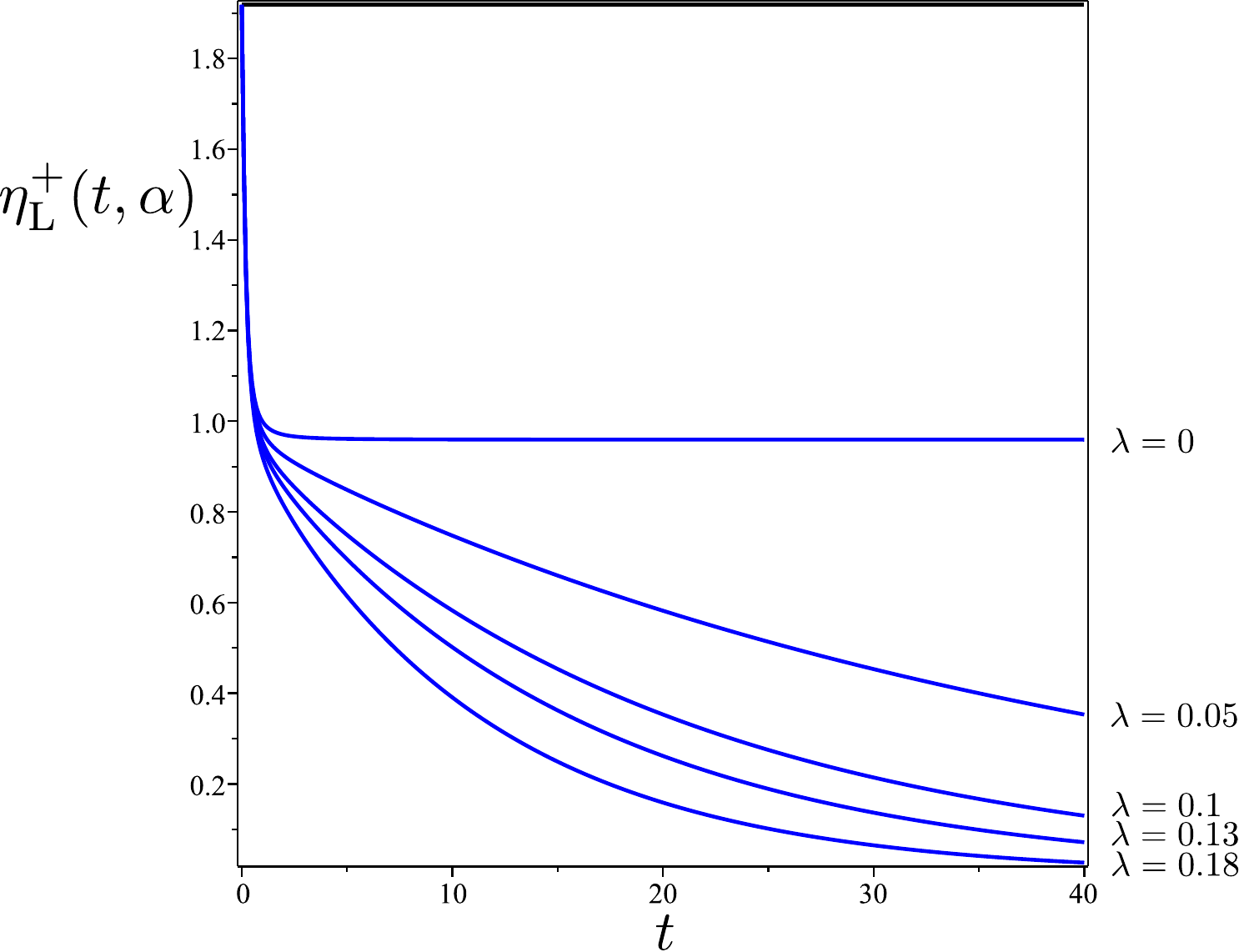}
	\caption{Graph of $\etaLp(t,\alpha)$, as defined in \eqref{eq:lower_bound_eta_delta}, plotted as a function of $t > 0$ for $\kappa = \frac{\pi}{4}$, $\zeta = \frac{1}{2}$, $\alpha=0.21$, and a sequence of values of $\lambda = 0, \, 0.05, \, 0.1, \, 0.13, \, 0.18$. The function is strictly decreasing and remains positive throughout the domain $t > 0$. The value $\alpha\doteq0.21$ minimizes $\tc(\lambda,\alpha)$ for every selected $\lambda$.}
	\label{fig:eta_l_delta_graph}
\end{figure}
Blue lines in Fig.~\ref{fig:T_crit_arctan} show how the pre-critical time $\tch(\lambda)$  depends on $\lambda$ for fixed $\kappa$ and $\zeta$. In accord with physical intuition, $\tch(\lambda)$ is a monotonically increasing function of $\lambda$: as $\lambda$ increases, the wave undergoes a greater damping, resulting in a stronger progressive attenuation over time and a corresponding delay in forming shocks. Furthermore, for given $\lambda$, $\tch(\lambda)$ decreases as $\kappa$ decreases and increases as $\zeta$ increases: a decrease in $\kappa$ or an increase in $\zeta$ results indeed  in a wider or lower initial profile, suggesting that the kink's core propagates more slowly, thus delaying the occurrence of shocks.  Figure~\ref{fig:T_crit_arctan} also brings to effect the selection rule \eqref{integral_q_positive_approx}; the \emph{acceptable range} of $\lambda$ for the validity of our theory is determined graphically as the maximum interval over which the red line is above the blue one. The cases shown in Fig.~\ref{fig:T_crit_arctan} seem to indicate that the acceptable range of $\lambda$ increases when either the amplitude of the kink increases or the width decreases.

\begin{figure}[]
	\centering
	\begin{subfigure}[c]{0.35\linewidth}
		\centering
		\includegraphics[width=\linewidth]{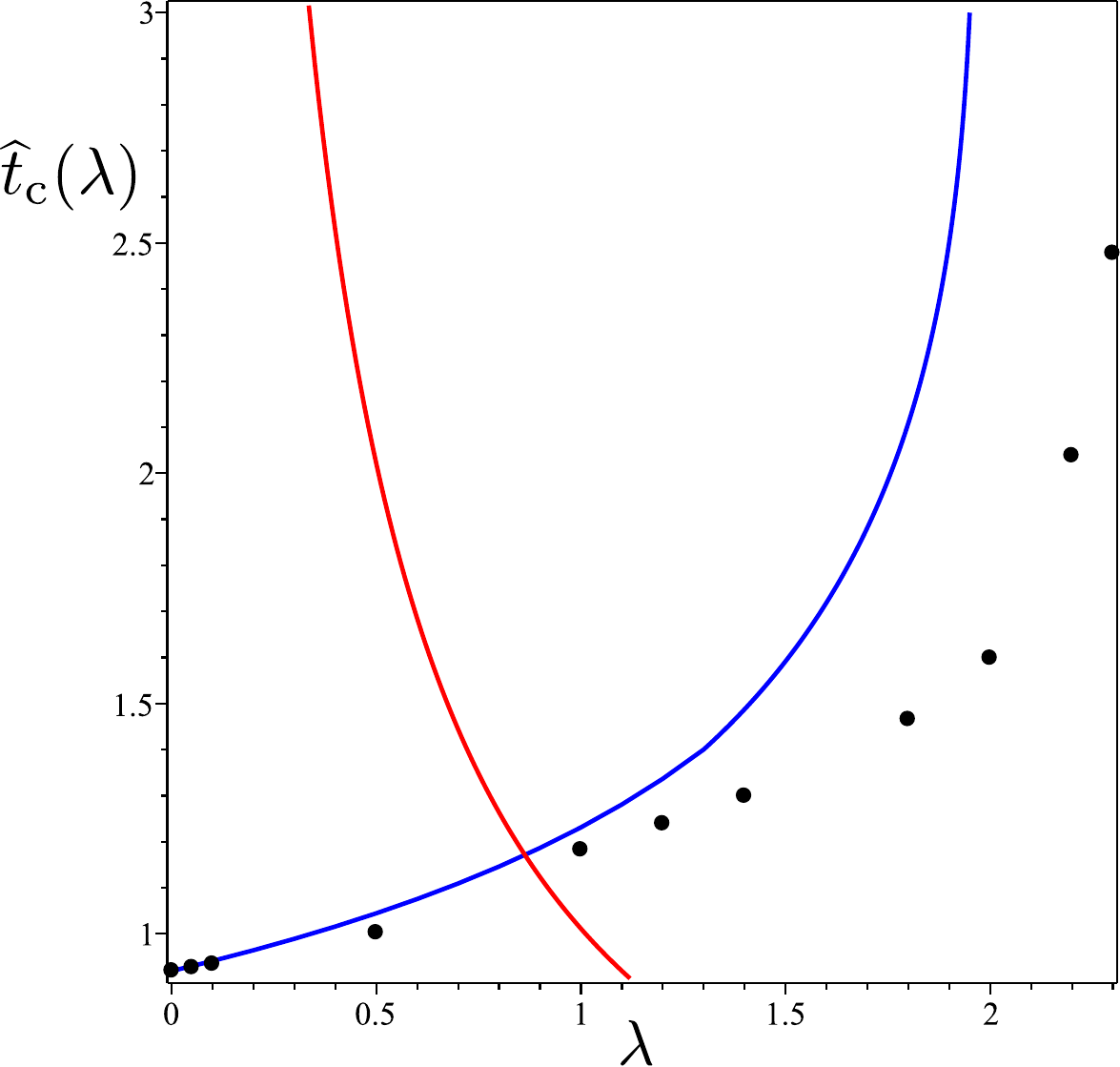}
		\caption{$\kappa=\pi/2$, $\zeta=1/2$. The acceptable range for $\lambda$ is $0\leq\lambda\leq0.87$.} 
		\label{fig:T_crit_arctanzeta}
	\end{subfigure}
	\begin{subfigure}[c]{0.35\linewidth}
		\centering
		\includegraphics[width=\linewidth]{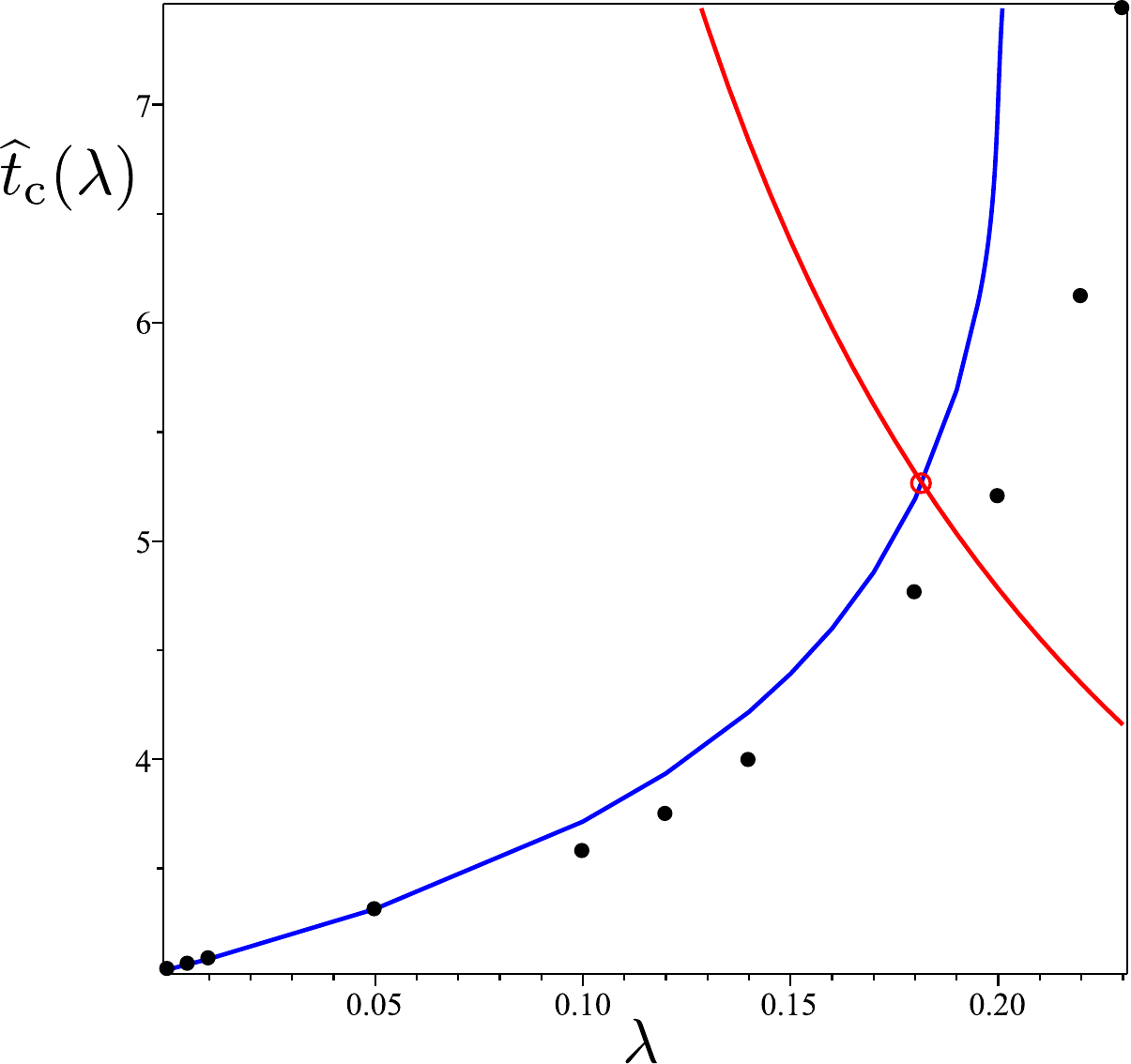}
		\caption{$\kappa=\pi/4$, $\zeta=1/2$. The acceptable range for $\lambda$ is $0\leq\lambda\leq0.18$.} 
		\label{fig:T_crit_arctankappa}
	\end{subfigure}
		\begin{subfigure}[c]{0.35\linewidth}
		\centering
		\includegraphics[width=\linewidth]{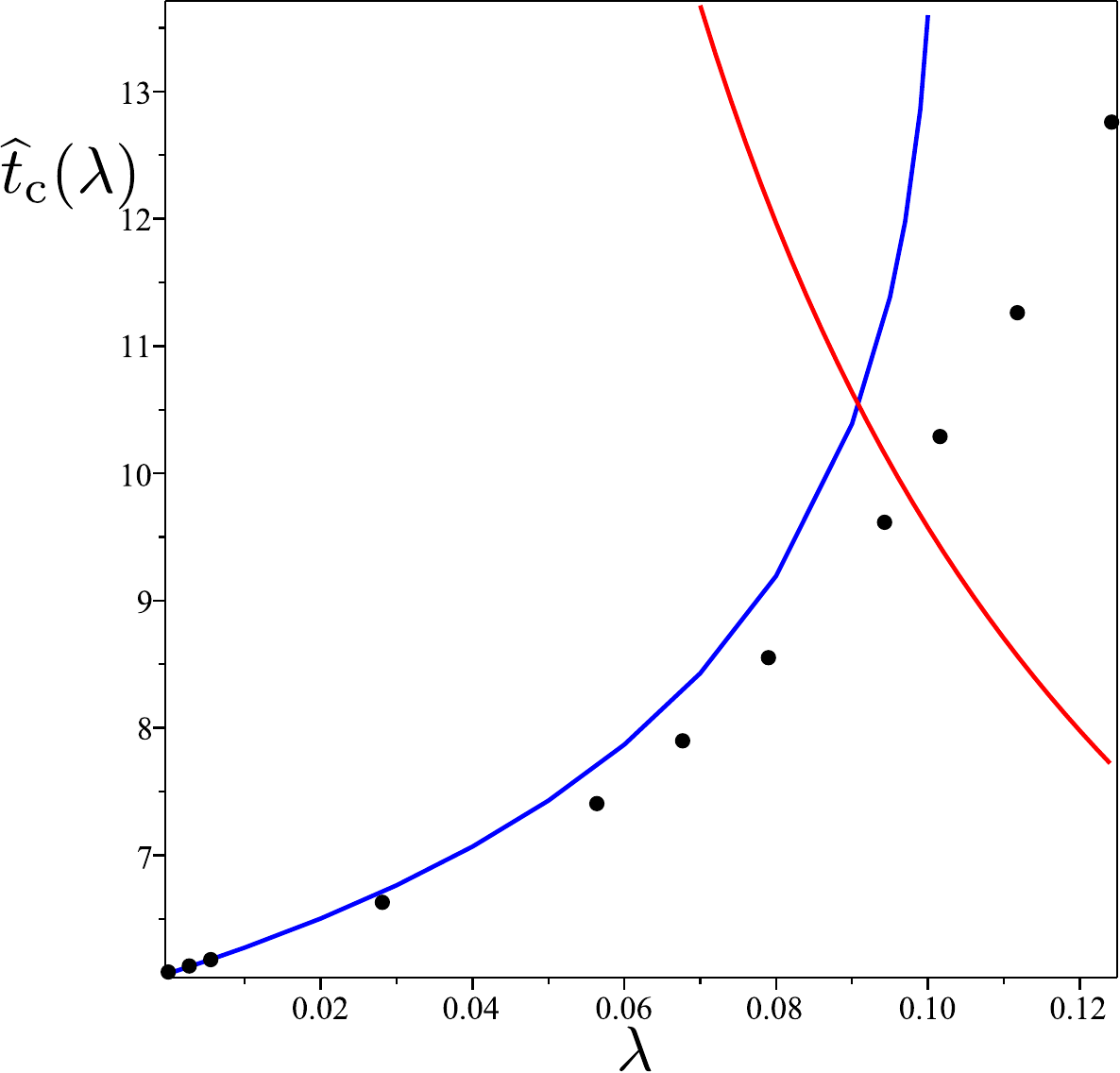}
		\caption{$\kappa=\pi/2$, $\zeta=1$. The acceptable range for $\lambda$ is $0\leq\lambda\leq0.09$.} 
		\label{fig:T_crit_arctankappa}
	\end{subfigure}
	\begin{subfigure}[c]{0.35\linewidth}
		\centering
		\includegraphics[width=\linewidth]{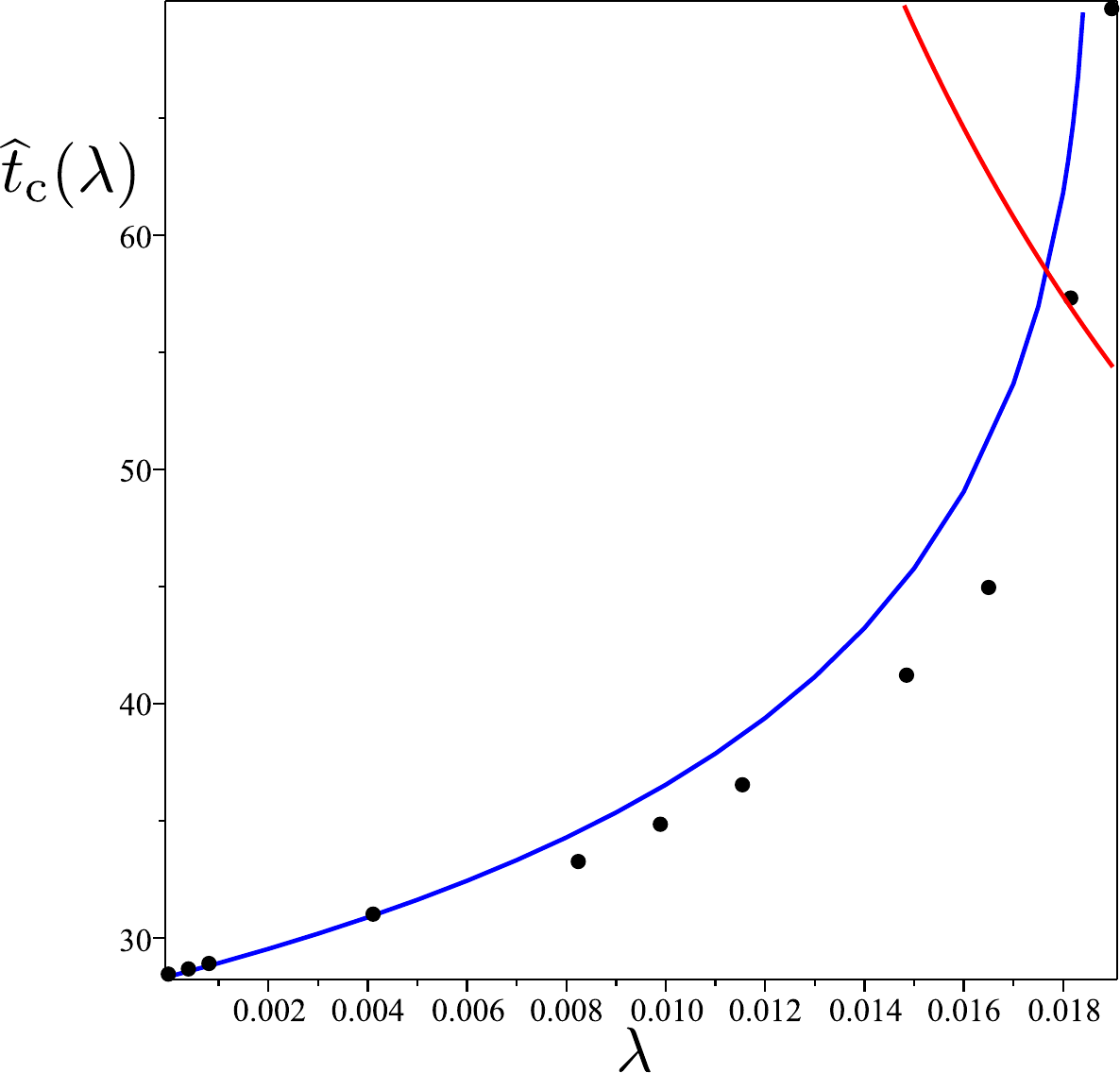}
		\caption{$\kappa=\pi/4$, $\zeta=1$. The acceptable range for $\lambda$ is $0\leq\lambda\leq0.018$.} 
		\label{fig:T_crit_arctankappa}
	\end{subfigure}
\caption{For $\lambda>0$, the time $t_\mathrm{c}(\lambda)$ in \eqref{eq:t_ast_definition} estimates the critical time $\tast$ at which a regular solution to the global Cauchy problem \eqref{eq:wave_system} subject to \eqref{eq:w_0_arctan} develpos a shock. For example, for $\lambda=0.18$, $\kappa=\pi/4$, and $\zeta=1/2$,  $t_\mathrm{c} \doteq 5.1$ and the infimum in \eqref{eq:t_ast_definition} is attained at $\alpha \doteq 0.21$ (see red dot in panel (b)). Red lines represent  the effect  of the selection rule \eqref{integral_q_positive_approx}: the only acceptable pre-critical times $\tch(\lambda)$ are those that fall below these lines.  Black dots represent  critical times $\tast$ computed numerically for various values of $\lambda$.}
	\label{fig:T_crit_arctan}
\end{figure}

Numerical solutions of the Cauchy problem \eqref{eq:wave_system} closely corroborate our theoretical predictions. Black dots in Fig.~\ref{fig:T_crit_arctan} represent critical times $\tast$ computed numerically for different values of $\lambda \geq 0$, with fixed $\kappa$ and $\zeta$. Within the range of validity of our theory, these numerical values are in close agreement with the estimated values of the pre-critical times $\tch(\lambda)$.

We paid special attention to the case where in \eqref{eq:w_0_arctan} $\kappa=\pi/4$ and $\zeta=1/2$;  for these specific values of $\kappa$ and $\zeta$,  $\lambda=0.18$ is the upper acceptable limit for the validity of  our theory. The initial profile generates two symmetric waves propagating in opposite directions. We computed the conserved quantity $M$ and the decaying quantity $E$ introduced in Propositions~\ref{prop:mass_conservation} and \ref{prop:dissipation_law}, respectively, and used them to monitor the accuracy of our numerical solutions. Our calculations indicate the existence of a critical time, estimated as $t^\ast_\lambda\doteq4.7$, at which the solution exhibits a singularity. Fig.~\ref{fig:num_sol_arctan} illustrates a typical  numerical solution $\twist(t,x)$ and its spatial derivatives $\partial_{x}\twist(t,x)$ and $\partial_{xx}\twist(t,x)$ for  a sequence of times $t$ in the interval $[0,\tast)$. A  snapshot at $t=\tast$ is shown in Fig.~\ref{fig:num:sol_t_crit_arctan}.
The observed behavior is in good agreement with our theory: a shock is formed in a finite time, at which $\partial_{x}\twist$ becomes discontinuous, and  second derivatives diverge.

The critical time identified numerically agrees with our  theoretical upper estimate $t_\mathrm{c}(\lambda)\doteq5.1$; the infimum in \eqref{eq:t_ast_definition} is correspondingly attained for $\alpha\doteq0.21$. 
\begin{figure}[]
	\centering
	\begin{subfigure}[c]{0.35\linewidth}
		\centering
		\includegraphics[width=\linewidth]{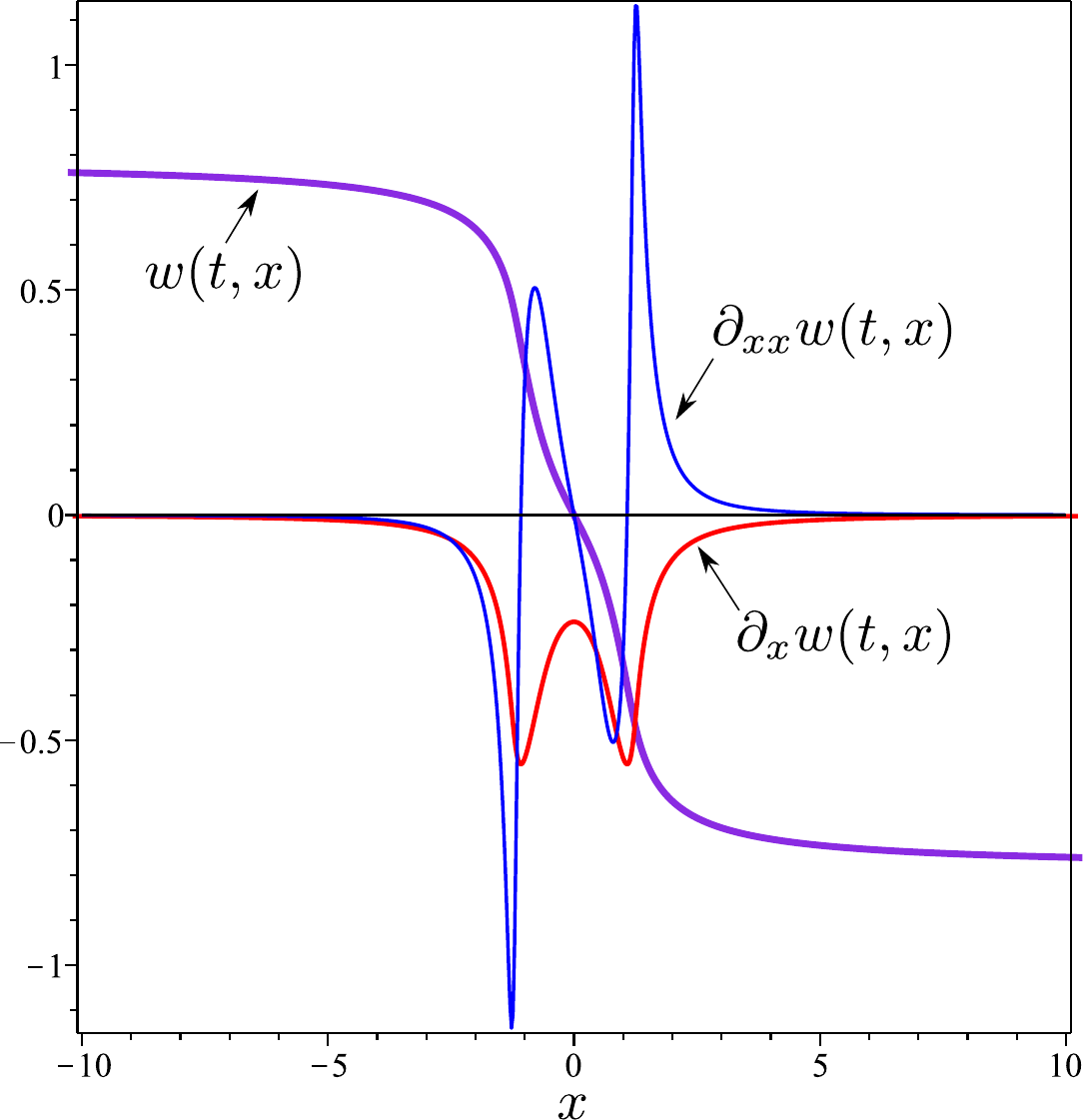}
		\caption{$t=0.9$} 
		\label{fig:t_015}
	\end{subfigure}
	\quad
	\begin{subfigure}[c]{0.35\linewidth}
		\centering
		\includegraphics[width=\linewidth]{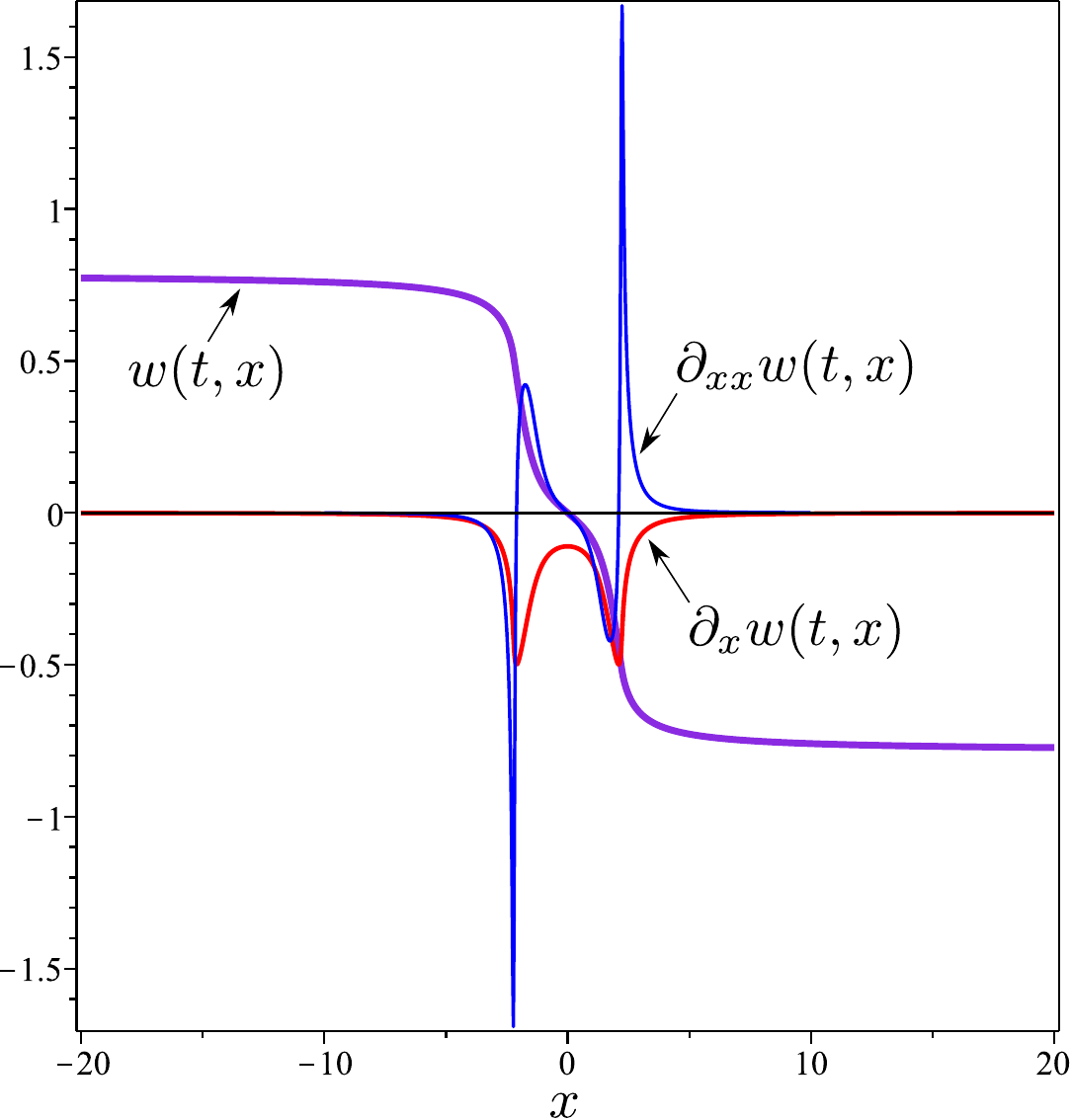}
		\caption{$t=1.8$} 
		\label{fig:t_03}
	\end{subfigure}
	\begin{subfigure}[c]{0.35\linewidth}
		\centering
		\includegraphics[width=\linewidth]{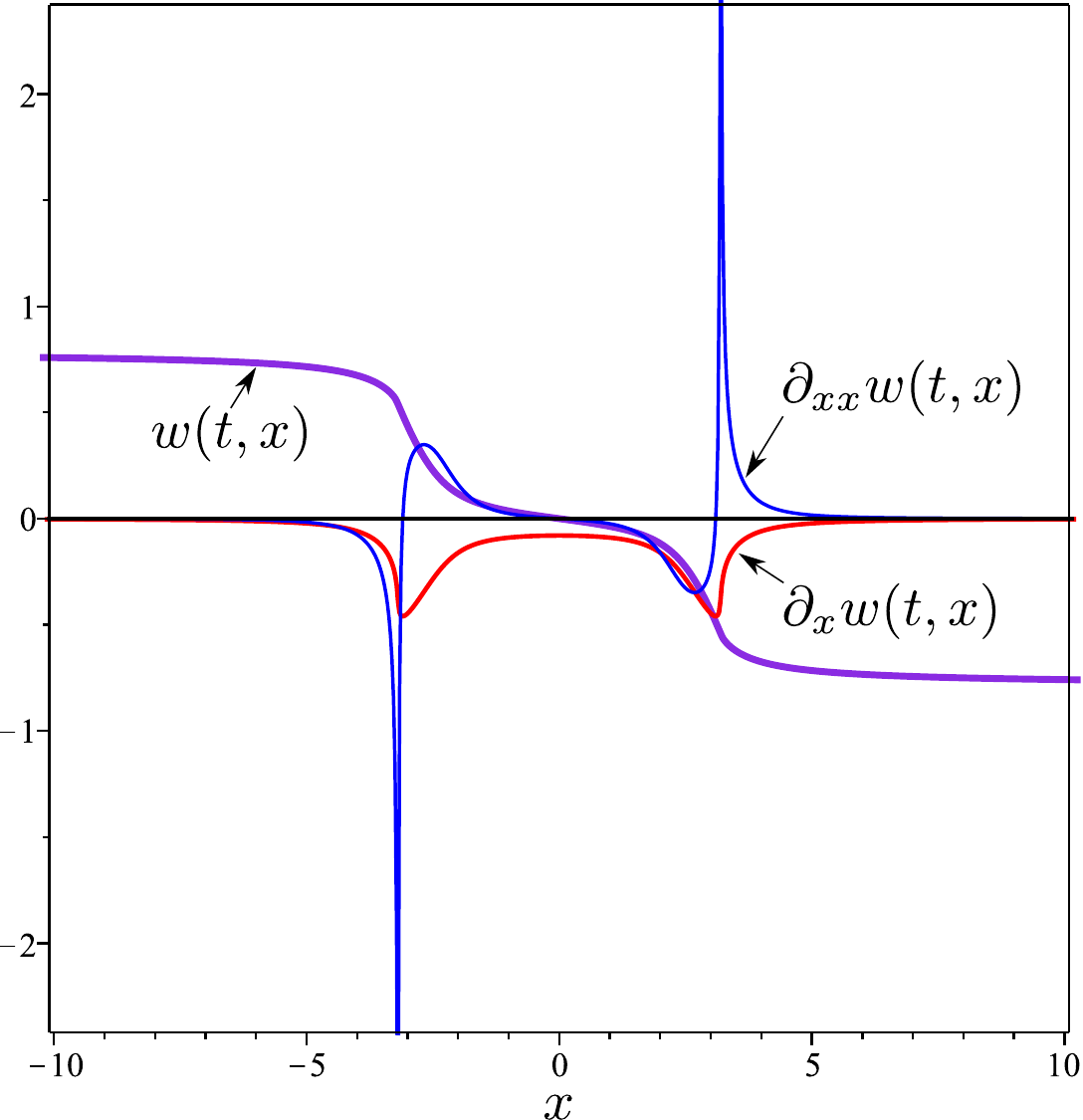}
		\caption{$t=2.7$} 
		\label{fig:t_27}
	\end{subfigure}
	\quad
	\begin{subfigure}[c]{0.35\linewidth}
		\centering
		\includegraphics[width=\linewidth]{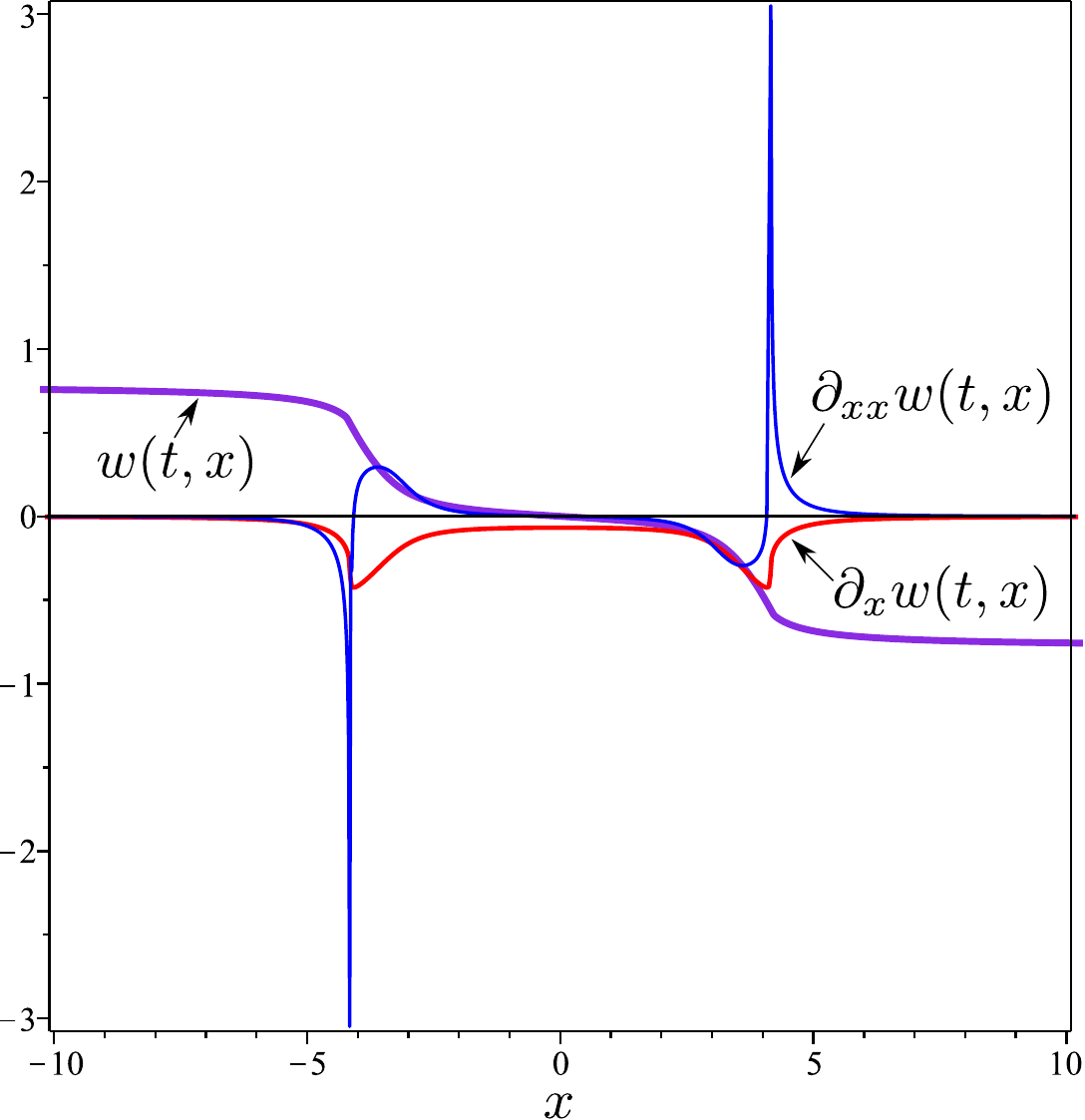}
		\caption{$t=3.6$} 
		\label{fig:t_36}
	\end{subfigure}
\caption{Graphs of the numerical solution $\twist(t,x)$ and its spatial derivatives $\partial_{x}\twist(t,x)$ and $\partial_{xx}\twist(t,x)$ for the initial profile $w_0$ in \eqref{eq:w_0_arctan} with $\kappa = \pi/4$, $\zeta=1/2$, and $\lambda=0.18$. The time sequence suggests that $\partial_{xx}w$ tends to develop two antisymmetric spikes, whereas $\partial_{x}w$ tends to develop symmetric jumps.}
	\label{fig:num_sol_arctan}
\end{figure}

\begin{figure}[]
			\centering
		\includegraphics[width=0.4\linewidth]{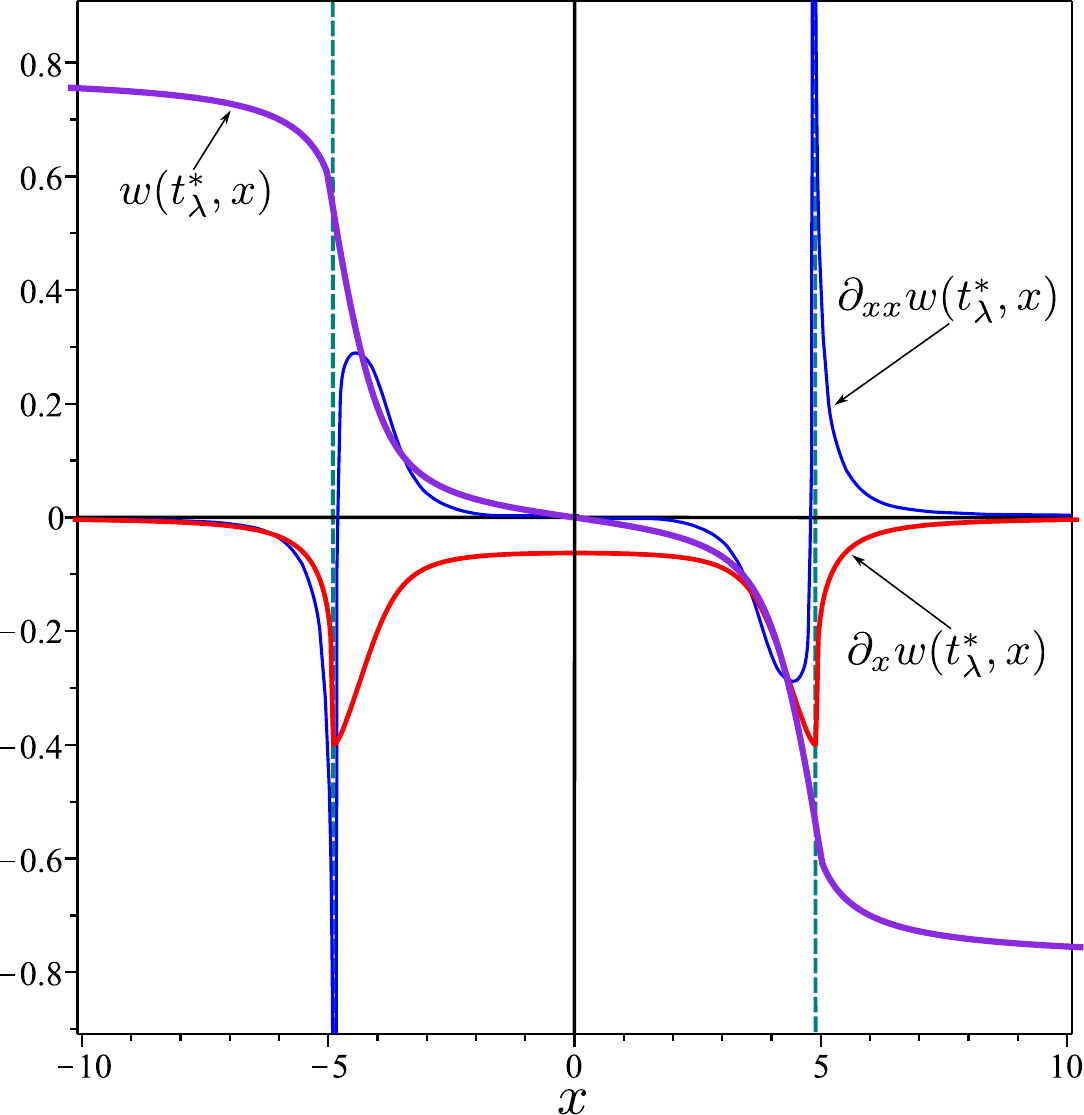}
		\caption{Singularity emerging  at the critical time $\tast\doteq4.7$ along  the solution of \eqref{eq:wave_system} with initial profile \eqref{eq:w_0_arctan} for $\kappa = \pi/2$, $\zeta=1/2$, and $\lambda=0.18$. The figure shows the numerical solution $\twist(\tast, x)$, along with its spatial derivatives $\partial_{x}\twist(\tast, x)$ and $\partial_{xx}\twist(t^\ast_\lambda, x)$, evaluated at the critical time $\tast \doteq 4.7$.}
	\label{fig:num:sol_t_crit_arctan}
\end{figure}

\section{Conclusions}\label{sec:conclusion}
The relaxation dynamics of liquid crystals involves only the motion of the director which represents the average molecular orientation; no hydrodynamic flow is entrained. The director motion is hampered by rotational viscosity and, in ordinary nematic liquid crystals, it drives the fluid \emph{smoothly} towards equilibrium.

Such a state of affairs, common to many dissipative systems, does \emph{not} apply to chromonic liquid crystals. They are as dissipative as ordinary nematic liquid crystals, but their elasticity is better represented by a free-energy density \emph{quartic} in the director gradient, instead of quadratic, as in the celebrated Oseen-Frank formula for ordinary nematics. This peculiar aspect of chromonics affects deeply the propagation of twist waves. It was proved in \cite{paparini:singular} that in the inviscid limit, where all sources of dissipation are neglected, a twist wave generally develops a singularity, giving rise to a \emph{shock} wave in finite time.

The present study, which is the ideal sequel of \cite{paparini:singular}, was concerned with the possible regularizing effect of dissipation. We asked whether a finite degree of dissipation could systematically prevent singularities from arising, thus leading smoothly the system to stillness.

We answered this question for the \emph{negative}: We provided a criterion for the choice of the initial data that guarantees the formation of a shock wave and estimated the critical time for its manifestation. This criterion was also subject to a numerical validation; for a class of initial director profiles our estimate for the critical time proved to be accurate, more so for small values of the dimensionless parameter $\lambda$ that weights dissipation against elasticity.

This is the main achievement of our paper, but it came at a price. A selection rule had to be enforced on the time interval where our analytical method is applicable. Such a rule effectively reduces the range of acceptable values of $\lambda$ in a manner influenced by the initial director profile.

Although, we provided a tool to probe the range of confidence of our method, this limitation can hardly be ignored. It is likely to stem from the role played here by the Riemann functions: they are no longer invariant in the dissipative setting, but they are still the main players in our analysis. We might have pushed this method to the extreme and paid a price in generality.

We have established that dissipation does nor prevent shocks to develop in chromonic twist waves, but an important question remains unanswered: is there a critical value of $\lambda$ (possibly affected by the initial director profile) above which singularities are banned and twist waves fade away uneventfully? The answer to this might come from an approach different from the one taken here.

\begin{acknowledgements}
	Both authors are members of \emph{GNFM}, a branch of \emph{INdAM}, the Italian Institute for Advanced Mathematics. For the most part, this work was performed while S.P. was holding a Postdoctoral Fellowship with the Department of Mathematics of the University of Pavia. Their kind hospitality during later times, while this paper was completed, is gratefully acknowledged.
\end{acknowledgements}
\appendix

\section{Auxiliary Results}\label{sec:aux_res} 
In this Appendix, we collect a number of auxiliary results needed in the main text, including the proofs of Lemmata~\ref{lemma:r_l_bounded} and \ref{lemma:w_x_t_infty}, and that of Propositions~\ref{th:infinitesimal_compression_ratios} and ~\ref{prop:beta_t_alpha}.

We start by stating a comparison theorem (see Theorem~1.3 of \cite{teschl:ordinary}).
\begin{theorem}[Comparison Theorem]
If $g(t,x)$ is a function locally Lipschitz continuous in  $x$ and uniformly continuous  in $t$, and  if $x(t)$ and $y(t)$ are two differentiable functions such that
\begin{equation}
x(t_0) \leq y(t_0)\quad\text{and}\quad  \dot x(t)-g(t,x(t))\leq\dot y(t)-g(t,y(t))\quad\text{for every } t\in[t_0,t^*),
\end{equation}
then
\begin{equation}
x(t) \leq y(t)\quad \text{for every }t\in[t_0,t^*).
\end{equation}
\end{theorem}
To apply this theorem to \eqref{eq:comparison_v_zeta}, it suffices to take
\begin{equation}
	\label{eq:choice_g}
	g(t,x)=G(t)x^2,
\end{equation}
where $G(t):=A(t)^{-1}f(\eta(t,x_1(t,\alpha)))$. The function $g$ in \eqref{eq:choice_g} is clearly locally Lipschitz continuous in $x$  and uniformly continuous in $t$, given the regularity of the solution $\eta(t,x_1(t,\alpha))$ and its boundedness (by Lemma~\ref{lemma:r_l_bounded}).

\begin{proof}[Proof of Lemma~\ref{lemma:r_l_bounded}]
We only prove \eqref{eq:bounded_r_l_plus}. On the characteristic curves $x=x_1(t,\alpha_0)$ and $x=x_2(t,\beta_0)$ with $\alpha_0$ and $\beta_0$ selected so that these curves pass through the points $(t_0,x_0)$ and $(t_0,y_0)$, respectively, by \eqref{eq:A_r_l}, we obtain that
\begin{equation}
\label{eq:sum_r_l}
A(t)\left(r(t_0,x_0)-\ell(t_0,y_0)\right)=r_0(\alpha_0)+\ell_0(\beta_0)+\frac{\lambda}{2}\int_0^{t_0}A(\tau)\left[r(\tau,x_2(\tau,\beta))-\ell(\tau,x_1(\tau,\alpha))\right]\dd\tau.
\end{equation}
By the definition of $\varphi=\varphi(t,x_0,y_0)$ in Lemma~\ref{lemma:r_l_bounded}, this equation is equivalent to
\begin{equation}
\label{eq:phi_r_l}
A(t)\varphi(t,x,y)=\varphi(0,\alpha,\beta)-\frac{\lambda}{2}\int_0^{t}A(\tau)\varphi(\tau,x_1(\tau,\alpha),x_2(\tau,\beta))\dd\tau,
\end{equation}
and by taking the $L^{\infty}$-norm of both sides of \eqref{eq:phi_r_l} with $(x,y)\in\mathbb{R}^2$, we arrive at
\begin{equation}
\label{eq:phi_r_l_norm}
A(t)||\varphi(t,\cdot,\cdot)||_{L^{\infty}(\mathbb{R}^2)}\leq||\varphi(0,\cdot,\cdot)||_{L^{\infty}(\mathbb{R}^2)}+\frac{\lambda}{2}\int_0^{t}A(\tau)||\varphi(\tau,\cdot,\cdot)||_{L^{\infty}(\mathbb{R}^2)}\dd\tau.
\end{equation}
By the classical Gr\"onwall's inequality, it then follows that 
\begin{equation}
\label{eq:phi_r_l_norm1}
A(t)||\varphi(t,\cdot,\cdot)||_{L^{\infty}(\mathbb{R}^2)}\leq A(t)||\varphi(0,\cdot,\cdot)||_{L^{\infty}(\mathbb{R}^2)},
\end{equation}
which implies inequality \eqref{eq:bounded_r_l_plus} in the main text. Inequality \eqref{eq:bounded_r_l_minus} can be proved in precisely the same way.
\end{proof}

\begin{proof}[Proof of Lemma~\ref{lemma:w_x_t_infty}]
Since $\partial_{x}w(t,x)$ can be expressed through \eqref{eq:first_order_unknowns} and \eqref{eq:u_1_rl} in terms of $\eta=r-\ell$,
\begin{equation}
\label{eq:w_x_eta}
\partial_{x}w(t,x)=-L^{-1}\left(\frac12\eta(t,x)\right),
\end{equation}
we prove Lemma~\ref{lemma:w_x_t_infty} by studying the behaviour of $\eta(t,x)$ for $|x|\to\infty$. We set
\begin{subequations}
\label{eq:S_I_t}
\begin{align}
S^\pm(t):=&\limsup_{x\to\pm\infty}\eta(t,x),\\
I^\pm(t):=&\liminf_{x\to\pm\infty}\eta(t,x),
\end{align}
\end{subequations}
and we consider the characteristic curves $x=x_1(t,\alpha_0)$ and $x=x_2(t,\beta_0)$ with $\alpha_0$ and $\beta_0$ selected so that these curves meet at a given point $(t_0,x_0)$; as long as the solution remains regular, they are uniquely identified (see Proposition~\ref{prop:beta_t_alpha} in the main text). Then, by \eqref{eq:A_r_l}, we can express $\eta(t_0,x_0)$ as 
\begin{equation}
\label{eq:sum_r_l_u1}
A(t)\eta(t_0,x_0)=r_0(x_1(0,\alpha_0))-\ell_0(x_2(0,\beta_0))+\frac{\lambda}{2}\int_0^{t_0}A(\tau)\left(r(\tau,x_2(\tau,\beta_0))-\ell(\tau,x_1(\tau,\alpha_0))\right)\dd\tau.
\end{equation}

Since for any given $t$, $\lim_{\alpha\to\pm\infty}x_1(t,\alpha)=\lim_{\beta\to\pm\infty}x_2(t,\alpha)$, by the arbitrariness of $(t_0,x_0)$, it follows from \eqref{eq:sum_r_l_u1} that
\begin{align}
\label{eq:A_t_S_t}
A(t_0)&S^\pm(t_0)=\limsup_{|\alpha_0|,|\beta_0|\to\infty}\left(r_0(x_1(0,\alpha_0))-\ell_0(x_2(0,\beta_0))\right)+\frac{\lambda}{2}\limsup_{x_0\to\pm\infty}\int_0^{t_0}A(\tau)\left(r(\tau,x_2(\tau,\beta_0))-\ell(\tau,x_1(\tau,\alpha_0))\right)\dd\tau\nonumber\\
&\leq r_0(\pm\infty)-\ell_0(\pm\infty)+\frac{\lambda}{2}\lim_{R\to\infty}\int_0^{t_0}\dd\tau A(\tau)\begin{cases}
\sup\left\{r(\tau,x_2(\tau,\beta_0))-\ell(\tau,x_1(\tau,\alpha_0)), x_0\geq R\right\} & \hbox{if } x_0\to+\infty\\
\sup\left\{r(\tau,x_2(\tau,\beta_0))-\ell(\tau,x_1(\tau,\alpha_0)), x_0\leq -R\right\} & \hbox{if } x_0\to-\infty
\end{cases}\nonumber\\
&\leq r_0(\pm\infty)-\ell_0(\pm\infty)+\int_0^{t_0} A(\tau)S^\pm(\tau) \dd\tau,
\end{align}
where the last inequality is obtained through the Lebesgue Convergence Theorem,
since the function $r(\tau,x_2(\tau,\beta_0))-\ell(\tau,x_1(\tau,\alpha_0))$ is bounded for every $\tau\in[0,t_0)$ and for every $t_0\in[0,\tast)$ by \eqref{eq:bounded_r_l_minus}. By the classical Gr\"onwall's inequality, we then obtain that
\begin{equation}
\label{eq:S_t_bounded}
S^\pm(t_0)\leq r_0(\pm\infty)-\ell_0(\pm\infty).
\end{equation}
In a similar way, we also prove that
\begin{equation}
\label{eq:I_t_lower_bound}
I^\pm(t_0)\geq r_0(\pm\infty)-\ell_0(\pm\infty).
\end{equation}
Since $S^\pm(t_0)\geq I^\pm(t_0)$ by definition, \eqref{eq:S_t_bounded} and \eqref{eq:I_t_lower_bound} imply that
\begin{equation}
\label{eq:L_u1_infinity}
S^\pm(t_0)=I^\pm(t_0)=\lim_{x_0\to\pm\infty}\left(r(t_0,x_0)-\ell(t_0,x_0)\right)=r_0(\pm\infty)-\ell_0(\pm\infty),
\end{equation}
which, by \eqref{eq:w_x_eta}, proves \eqref{eq:lim_infinity_wx}.

Similarly, equation \eqref{eq:lim_infinity_wt} is proved by writing
\begin{equation}
\partial_{t}w(t,x)=\frac12(r(t,x)+\ell(t,x)),
\end{equation}
which follows from \eqref{eq:riemann_system}.
\end{proof}
\begin{proof}[Proof of Proposition~\ref{th:infinitesimal_compression_ratios}]
By differentiating  both sides of the first equation in \eqref{eq:x_1} with respect to $\alpha$, we find that
\begin{equation}
\label{eq:d_alpha_tx_1}
\partial_\alpha\left(\dfrac{\dd x_1}{\dd t}(t,\beta)\right)=\partial_\alpha k(\eta(t,x_1(t,\alpha)))=k'\left(\eta(t,x_1(t,\alpha))\right)\left(\partial_{x}r(t,x_1(t,\alpha))-\partial_{x}\ell(t,x_1(t,\alpha))\right)\partial_\alpha x_1(t,\alpha).
\end{equation}
Letting
\begin{equation}
\label{eq:H_def}
H(\eta):=\int_0^\eta\dfrac{k'(y)}{2k(y)}\dd y=\frac{1}{2}\ln k(\eta),
\end{equation}
where also \eqref{eq:k_l_r} has been used, by \eqref{eq:partial_plus} we arrive at
\begin{equation}
\label{eq:d_t_H}
\frac{\dd}{\dd t}H(\eta(t,x_1(t,\alpha)))=\frac{k'(\eta(t,x_1(t,\alpha)))}{2k(\eta(t,x_1(t,\alpha)))}\partial^+\eta=-k'(\eta(t,x_1(t,\alpha)))\partial_{x}\ell(t,x_1(t,\alpha)).
\end{equation}
Then, \eqref{eq:d_alpha_tx_1} reduces to
\begin{equation}
\label{eq:ode_x_1_alpha}
\frac{\dd}{\dd t}c_1(t,\alpha)=\left[k'(\eta(t,x_1(t,\alpha)))\partial_{x}r(t,x_1(t,\alpha))+\frac{\dd}{\dd t}H(\eta(t,x_1(t,\alpha)))\right]c_1(t,\alpha),
\end{equation}
which can be set in the form
\begin{equation}
\label{eq:ode_gen}
\dot{c}_1-\dot{h} c_1=0,
\end{equation}
where, also by \eqref{eq:y_def},
\begin{equation}
\label{eq:ode_ingredients}
h(t,\alpha):=\int_0^tA(\tau)^{-1}f(\eta(\tau,x_1(\tau,\alpha)))p(\tau,x_1(\tau,\alpha))\dd\tau+H(\eta(t,x_1(t,\alpha))).
\end{equation}
The solution of \eqref{eq:ode_gen} subject to the initial condition $c_1(0,\alpha)=1$ is the following,
\begin{equation}
\label{eq:solution_ode_gen}
c_1(t,\alpha)=\mathrm{e}^{h(t,\alpha)-h(0,\alpha)}.
\end{equation}
Since, by \eqref{eq:ode_ingredients} and \eqref{eq:IC_riemann}, $h(0,\alpha)=H(r_0(\alpha)-\ell_0(\alpha))=H(2r_0(\alpha))$, \eqref{eq:solution_ode_gen} delivers  equation \eqref{eq:infinitesimal_compression_ratios1_expr} in the main text.

Equation \eqref{eq:infinitesimal_compression_ratios2_expr} is obtained in a similar way.
\end{proof}
\begin{proof}[Proof of Proposition~\ref{prop:beta_t_alpha}]
The existence of the unique solution $\beta=\beta(t,\alpha)$ is a consequence of the fact that the Cauchy problem
\begin{equation}
\label{eq:backward_characteristic}
\begin{cases}
\frac{\dd y_2}{\dd \tau}(\tau,t,\alpha)=-k(\tau,y_2(\tau,t,\alpha)), \\
y_2(t,t,\alpha)=x_1(t,\alpha),
\end{cases}
\end{equation}
has a unique solution $y_2(\tau,t,\alpha)$ for every $\tau\in[0,\tast)$. We set $\beta(t,\alpha):=y_2(0,t,\alpha)$ and $y_2(\tau,t,\alpha):=x_2(\tau,\beta(t,\alpha))$, where $x_2$ is  solution of \eqref{eq:x_2}. Thus, $\beta(t,\alpha)$ represents the point from which the backward characteristic $x_2$ starts. By \eqref{eq:x_2},  we conclude that $\beta(t,\alpha)>\alpha$ for every $t\in(0,\tast)$ and $\beta(0,\alpha)=\alpha$.
By differentiating both sides of equation \eqref{eq:beta_of_alpha_implicit_definition} with respect to $t$, and using \eqref{eq:x_1} and \eqref{eq:x_2}, we obtain \eqref{eq:d_t_beta}.

Next, by integrating both sides of \eqref{eq:x_1} with respect to $t$ and using \eqref{eq:estimates}, we arrive at
\begin{subequations}
\label{eq:upperbound_x1_x2}
\begin{align}
x_1(t,\alpha)&=\alpha+\int_0^{t}k(r(\alpha)-\ell(\tau,x_1(\tau,\alpha)))\dd\tau\leq\alpha +\delta t,\label{eq:upperbound_x1}\\
x_2(t,\beta(t,\alpha))&=\beta(t,\alpha)-\int_0^{t}k(r(\tau,x_2(\tau,\beta(t,\alpha)))-\ell_0(\beta(t,\alpha)))\dd\tau\geq \beta(t,\alpha) -\delta t. \label{eq:upperbound_x2}
\end{align}
\end{subequations}
By \eqref{eq:beta_of_alpha_implicit_definition} and \eqref{eq:upperbound_x1}, \eqref{eq:upperbound_x2} gives 
\begin{equation}
\label{eq:upperbound_beta}
\beta(t,\alpha)\leq\alpha+2\delta t,
\end{equation} 
which is the first property stated in \eqref{eq:beta_confined}. The other property can be established in a similar way by use of \eqref{eq:x_2}.
\end{proof}


%

\end{document}